\documentclass[10pt, reqno]{amsart}
\usepackage[latin1]{inputenc}
\usepackage{amsmath}
\usepackage{amsfonts}
\usepackage{amssymb}
\usepackage{graphicx}
\usepackage{svg}
\usepackage{mathtools}
\usepackage{caption}
\usepackage{enumerate}
\usepackage{tikz}
\usepackage{circuitikz}
\usepackage{subcaption}

\usepackage{amsthm}

\newcommand{\cP}{\mathcal{P}}
\newcommand{\absorbing}{\mathbf{e}}

\newcommand{\rate}[1]{e^{-c_1 #1}}

\newcommand{\fastDomain}{\chi}
\newcommand{\numAbs}{L}
\newcommand{\seqSpace}{\mathcal{L}}

\newcommand{\E}{\mathbb{E}}

\newcommand{\R}{\mathbb{R}}

\newcommand{\fastProcessSeqCont}{V^{\lambda, \mathbf{x}, \mathbf{v}, l}}
\newcommand{\fastProcessSeqLaw}{\mu^{\mathbf{x}, \mathbf{v}, l}}
\newcommand{\fastProcessSeqContLaw}{\mu^{\lambda, \mathbf{x}, \mathbf{v}, l}}

\newcommand{\LawLimitPosDep}{\tilde{\mu}_{\infty}^{ \mathbf{x},\mathbf{v}}}

\newcommand{\fastProcessSeq}{V^{\mathbf{x}, \mathbf{v}, l}}
\newcommand{\fast}{V^{\lambda, \mathbf{x}, \mathbf{v}}}
\newcommand{\fastFrozen}{\tilde{V}^{ \lambda, \mathbf{x}, \mathbf{v}}}
\newcommand{\fastFrozenDisc}{\tilde{V}^{\mathbf{x}, \mathbf{v}}}

\newcommand{\fastFrozenLaw}{\tilde{\mu}^{\lambda, \mathbf{x}, \mathbf{v}}}
\newcommand{\fastFrozenDiscLaw}{\tilde{\mu}^{ \mathbf{x}, \mathbf{v}}}

\newcommand{\fastProcessMat}{P}
\newcommand{\fastProcessMatAb}{P^{Ab}}
\newcommand{\fastProcessMatNAb}{P^{NA}}

\newcommand{\slowProcess}[2]{X^{#1}_{#2}}

\newcommand{\avgSol}{\mathcal{T}}

\newcommand{\bv}{\mathbf{v}}
\newcommand{\bx}{\mathbf{x}}
\newcommand{\bz}{\mathbf{z}}
\newcommand{\slow}{X_t^{\lambda, \bx, \bv}}
\newcommand{\fastuncoupled}{V_t^{\lambda,\bv}}
\newcommand{\fastcoupled}{V_t^{\lambda, \bx, \bv}}
\newcommand{\statespace}{\chi}
\newcommand{\slowbar}{\bar{X}_t^{\bx}}
\newcommand{\slowbarv}{\bar{X}_t^{\bx, \bv}}
\newcommand{\frozenintro}{\tilde{V}_t^{\lambda, \bx, \bv}}

\newcommand{\absorbingErg}[2]{\absorbing^{#1}_{#2}}

\newcommand{\transProbsT}[1]{P_{#1}}

\newcommand{\transProbs}[1]{\cP_{#1}}

\usepackage{etoolbox}
\patchcmd{\endproof}
  {\Qed}
  {\hfill \Qed}

\newtheorem{theorem}{Theorem}[section]
\newtheorem{lemma}{Lemma}[section]
\newtheorem{assumption}{Assumption}
\newtheorem{prop}{Proposition}[section]

\theoremstyle{remark}

\newtheorem{remarkx}[theorem]{Note}

\newtheorem{notex}[theorem]{Note}
\newenvironment{Note}
{\pushQED{\qed}\remarkx}
{\hfill\popQED\endnotex}

\begin{document}
\author{ B. D. Goddard$^{(1)}$, M. Ottobre$^{(2)}$, K. Painter$^{(3)}$  and I. Souttar$^{(4)}$}
	\address{(1) School of Mathematics and Maxwell Institute for Mathematical Sciences, University of
Edinburgh, Edinburgh EH9 3FD, UK. b.goddard at ed.ac.uk}
	\address{(2) Maxwell Institute for Mathematical Sciences and Mathematics Department, Heriot-Watt University, Edinburgh EH14 4AS, UK. m.ottobre at hw.ac.uk}

\address{(3) Inter-university Department of Regional and Urban Studies and Planning, Politecnico di Torino, Torino 10125, Italy, kevin.painter at polito.it}

\address{(4) Maxwell Institute for Mathematical Sciences and Mathematics Department, Heriot-Watt University, Edinburgh EH14 4AS, UK. is79 at hw.ac.uk}

\title{On the study of slow-fast dynamics,  when the fast process has multiple invariant measures}
\date{}
\maketitle

\begin{abstract}
Motivated by applications to mathematical biology, we study  the averaging problem for slow-fast systems, {\em in the  case in which the fast dynamics is a stochastic process with multiple invariant measures}. We consider both the case in which the fast process is decoupled from the slow process and  the case in which the two components are fully coupled. We work in  the setting in which the slow process evolves according to an Ordinary Differential Equation (ODE) and the fast process is a continuous time Markov Process with finite state space and show that, in this setting, the limiting (averaged) dynamics can be described  as a random ODE (that is, an ODE with random coefficients.)

{\bf Keywords.} Multiscale methods, Processes with multiple equilibria, Averaging, Collective Navigation, Interacting Piecewise Deterministic Markov Processes. 
   \smallskip
   
{\bf AMS Subject Classification 2020.}  34C29, 34D05, 34E10, 34F05, 37H10, 60K35.
\end{abstract}

\section{Introduction}
This paper is concerned with the study of the averaging problem for coupled slow-fast dynamics, in the case in which the slow dynamics is an Ordinary Differential Equation (ODE) and, more importantly,  the fast dynamics is a continuous time Markov process (CTMP) {\em with multiple invariant measures}.  This study  is inspired by problems in mathematical biology and in the modelling of collective behaviour (specifically it emerged from  \cite{painter} and related literature); in this work we focus on extracting the mathematical structure underlying such problems and on tackling it. In Section \ref{subsec:Heuristics} we will discuss more explicitly how the theory we develop here is relevant to such models. 

To explain the difficulty in considering the averaging problem when the fast process has multiple invariant measures, let us start by briefly recalling the classical averaging paradigm in the case in which the fast process has a unique invariant measure. The simplest case is the one in which the fast process evolves independently of the slow process, i.e. it is decoupled from the slow process.
 To fix ideas consider a slow-fast system $(\slow, \fastuncoupled)$, evolving as follows: the slow component evolves in $\R^N$ according to an ODE, 
\begin{align} \label{slowergodic}
   d\slow & = a(\slow, \fastuncoupled) dt\, ,  \quad X_0^{\lambda, \bx,\bv} = \bx \, ,
\end{align}
$a:\R^N \times \statespace \rightarrow \R^N$,  while the fast component $\fastuncoupled$ is a CTMP on a finite state space $\statespace$,  with $V_0^{\lambda,\bv} = \bv$. That is,  the process $\fastuncoupled$   changes state according to a given transition matrix $\cP$ every time a homogenous Poisson clock $\tau_t^{\lambda}$ with rate $\lambda$ goes off. Here $\cP = (\cP(\bv, \bv '))_{\bv, \bv' \in \statespace}$, where $\cP(\bv, \bv ')$ denotes the probability of going from $\bv$ to $\bv '$ and in short we say that $\fastuncoupled$ is a $(\tau_t^{\lambda}, \cP)$ -- CTMP.\footnote{We remark that here and throughout we will assume that the CTMP we deal with are càdlàg (right-continuous with left limits). Hence, if the Poisson clock goes off at time $s>0$, the value of $V_s^{\lambda, \bv}$ is the new value obtained by using the transition matrix $\cP$. The paths of the slow variable will instead be always continuous, thanks to our assumptions on $a$.} 

As $\lambda \rightarrow \infty$ the dynamics of $\fastuncoupled$ becomes faster and faster. 
 Assume now that the process $\fastuncoupled$  admits a unique invariant measure $\mu$ (i.e. it is {\em ergodic}), and 
 $$
 \fastuncoupled \rightharpoonup \mu, \quad \mbox{as } t \rightarrow \infty, \mbox{ for every fixed }
 \lambda>0 \,,   
 $$
where the above  convergence is meant in law.\footnote{Note that fixing $\lambda>0$ and letting $t \rightarrow \infty$ is here the same as fixing $t>0$ and letting $\lambda \rightarrow \infty$. } Then, as $\lambda \rightarrow \infty$,     the dynamics of the slow component $\slow$ converges weakly  to the so-called {\em averaged dynamics}, namely to the solution of the following ODE\footnote{With this setup the limit is deterministic so convergence is also in probability, but we stick to weak convergence because it is the mode of convergence that we will consider in this paper. }
	\begin{equation}\label{eqn:averaged}
		d\slowbar = \bar a(\slowbar) dt \,   \quad \bar X_0^{\bx} = \bx,    \end{equation}
where the {\em effective coefficient } $\bar a: \R^N \rightarrow \R^N$ is given by
$$
\bar a(\mathbf{z}):= \sum_{\mathbf y \in \statespace} a(\mathbf{z}, \mathbf y) \mu (\mathbf y) \,.
$$
We emphasize that, by ergodicity of the fast process,  the invariant measure $\mu$  does not depend on the initial state $\bv$ of the process $\fastuncoupled$, hence  the limiting averaged dynamics does not depend on $\bv$ either. 

In the above we have taken the slow process to evolve according to an ODE and the fast process to be a CTMP for two reasons:  because this is the setting we will consider in this paper and for ease of exposition. As is well known, however, the conclusion is true in much greater generality (see e.g. \cite{pavliotis2008multiscale} and more on the literature in Section \ref{subsec:relation to literature}). In particular, a similar result still holds when the fast process is coupled to the slow process.  In the setup we have just introduced this is  the case if, for example, the entries of the transition matrix  of the fast process depend not only on the current position of the fast process but also on the  position of  the  slow process. In this case we denote the fast process by $\fastcoupled$ rather than $\fastuncoupled$ (as its evolution  will depend  on the initial datum of the slow process as well) and  the evolution of $\fastcoupled$ is described by assigning a Poisson clock with rate $\lambda$ and  a family of transition matrices indexed by $\bx \in \R^N$, $\{\cP_{\bx}\}_{\bx \in \R^N} = \{\cP_{\bx}(\bv, \bv')_{\bv, \bv ' \in \statespace}\}_{\bx \in \R^N}$: if  the Poisson clock $\tau^{\lambda}_t$ goes off at time $s>0$, then the fast process uses the transition matrix $\cP_{X_s^{\lambda,\bx,\bv}}$ to determine the probability of transition to the next state. In short, we write that $\fastcoupled$ is a $(\tau_t^{\lambda}, \cP_{\slow})$-CTMP.  

When the slow and the fast processes are fully coupled, roughly speaking the  idea is that the slow process won't change much in the time it takes the fast process to reach equilibrium. So,   in order to understand the limiting dynamics,  one considers an intermediate process,  the so called `frozen process',  which we denote by $\fastFrozen_t$. This process is  obtained by `freezing'  the value of the slow component: in the setting described so far this is the CTMP $\fastFrozen_t$ on $\statespace$ obtained by fixing $\slow$ say to its initial value  $\slow=  \bx \in \R^N$,  and by then using the transition matrix $\cP_{\bx}$ to determine the transitions of the fast process when $\tau^{\lambda}$ goes off (that is, we  always use the same transition matrix $\cP_{\bx}$, irrespective of the time when the Poisson clock goes off). In other words, $\fastFrozen_t$ is a $(\tau_t^{\lambda}, \cP_{\bx})$--CTMP.  As customary, one can think of $\frozenintro$ as a family of processes depending on the parameter $\bx$.  Assuming that, for each $\bx$ fixed, the process  $\fastFrozen$  is ergodic  with (unique) invariant measure $\mu(\bx; \cdot)$,\footnote{We clarify that the notation $\mu(\bx; \cdot)$ means that for each $\bx$ fixed $\mu(\bx; \cdot)$ is a probability measure (in the second variable)}  as $\lambda \rightarrow \infty$ the slow process $\slow$ converges weakly  to the solution of the following ODE
\begin{equation} \label{intro:averagedcoupled}
d\slowbar = \bar a(\slowbar) dt, \qquad  \bar X_0^{ \bx, \bv} = \bx \, ,
\end{equation}
where this time the averaged coefficient $\bar a: \R^N \rightarrow \R^N$ is given by 
$$
\bar a(\bz) = \sum_{ \mathbf v \in \statespace} a(\bz,\mathbf v) d\mu( \bz; \mathbf v) \, \,.
$$
Note that the limiting invariant measure of the fast (frozen) process is unique for each $\bx$ fixed and, again by ergodicity of the frozen process, it does not depend on the initial state $\bv$. 

Whether the fast and slow component are fully coupled or not, we have used the {\em unique} invariant measure of the fast process (or, more precisely, of the frozen process) to determine the limiting drift $\bar a$, and hence the limiting dynamics. If the fast (frozen) process has more than one invariant measure (in the coupled case, more than one for each $\bx$ fixed),\footnote{Strictly speaking the number of invariant measures could depend on the $\bx$ we fix, we come back to this later.} then what is the evolution we obtain as $\lambda \rightarrow \infty$?  The precise answer to this question is stated in Theorem \ref{thm:averaging1Erg}, which is the main result of this paper. An informal statement of this theorem is as follows.  

\subsection{Informal statement of Main Result}\label{subsec:informalstatement} Let $(\slow, \fastcoupled) \in \R^N \times \statespace$ be a slow-fast process where the slow component $\slow$ evolves according to the following evolution equation, 
\begin{equation}\label{eqn:fullycoupled}
    dX^{\lambda,\mathbf{x}, \mathbf{v}}_t = a(X^{\lambda,\bx, \bv}_t, \fastcoupled) dt, \quad X^{\lambda,\bx, \bv}_0 = \bx,\, 
    V^{\lambda, \bx, \bv}_0= \mathbf{v} , 
\end{equation}
and $\fastcoupled$ is a $(\tau_t^{\lambda}, \cP_{\slow})$-CTMP.
Assume that the associated (family of)  frozen processes $\frozenintro$ -- or, equivalently, the (family of)  transition matrices $\cP_{\bx}$ -- are such that the following holds: under the dynamics of $\frozenintro$ the state space is partitioned into a transient set,  $\mathcal T$,  and $L$ ergodic classes,\footnote{By  ergodic class we mean a subset of state space which is  irreducible and positive recurrent. Hence, on each ergodic class $E^{(i)}$, the process admits a unique invariant measure, $\mu^{(i)}$, supported on that class. } $\statespace = \{E^{(1)} \cup  \dots \cup E^{(L)} \cup \mathcal T \}$, where the unions are disjoint and this decomposition is independent of $\bx$ (see Note \ref{note:comments on main result} and Section \ref{subsec:relation to literature} for comments on this assumption). Because of this state space decomposition, given any initial state $\bv$, there is probability one that $\frozenintro$ is absorbed in one of the ergodic classes. Let $q^{(i)}(\bx, \bv)$  be the probability that the frozen process  $\frozenintro$, starting at time zero from $\bv$, is absorbed in $E^{(i)}$ (so that $q^{(i)}(\bx, \bv)=1$ if $\bv \in E^{(i)}$).  We will also make assumptions which guarantee that the motion of the slow variable is such that, once the fast process $\fastcoupled$ enters an ergodic class, it will not leave it  (see Assumption \ref{ass:pErg} \eqref{ass:p:AbsErg} in  Section \ref{sec:ergClasses}). In this setting the limiting process is given by the following: let  $\zeta^{\mathbf{x},\mathbf{v}}$ be a discrete random variable which can take $L$ values and in particular takes the value $i \in \{1, \dots, L\}$ with probability $q^{(i)}(\bx, \bv)$. Draw a realization of $\zeta^{\mathbf{x},\mathbf{v}}$; if $\zeta^{\mathbf{x},\mathbf{v}}=i$ then $\slowbarv$ will evolve according to
\begin{equation}\label{eqn: informalstatementlimit}
d\slowbarv = \bar{a}^{(i)}(\slowbarv) dt, \quad \bar X_0^{\bx, \bv} = \bx \, 
\end{equation}
where 
\begin{equation}\label{eqn:avgedDriftErgClass}
    \bar{a}^{(i)}(\bz) = \sum_{ \bv \in E^{(i)}} a(\bz,\bv) \mu^{(i)}(\bz;\bv)  \,,
\end{equation}
and $\mu^{(i)}(\bz; \cdot)$ is the ergodic measure of the frozen process $\tilde{V}_{t}^{\lambda, \bz, \bv}$ supported on the class $E^{(i)}$.  The limiting drift  could also be seen as a function $\bar a= \bar a(\bz, \zeta^{\mathbf{x},\mathbf{v}})$, $\bar a: \R^N \times \{1, ..., L\} \rightarrow \R^N$. In other words, if $\zeta^{\mathbf{x},\mathbf{v}}$ is as in the above, then the limiting dynamics can be rewritten as
$$
d\slowbarv = \bar{a}(\slowbarv, \zeta^{\mathbf{x},\mathbf{v}}) dt, \quad \bar X_0^{\bx, \bv} = \bx \, ,
$$
with $\bar a(\bz, \zeta^{\mathbf{x},\mathbf{v}}) = \bar a^{(i)}(\bz)$ if $\zeta^{\mathbf{x},\mathbf{v}} =i$. For this reason we say that the limiting process is a random ODE.  The precise statement of this result is contained in Theorem \ref{thm:averaging1Erg}, which is the main (more general) result of this paper.

In this non-ergodic setting the method of multiscale expansions \cite{pavliotis2008multiscale} does not  seem to be directly applicable to obtain, even just formally, an expression for the limiting dynamics. Hence, besides technical aspects, it is important to develop an intuition on how to obtain the limiting equation. For this reason the rest of this paper is structured as follows. 

In Section \ref{subsec:Heuristics} we give a heuristic argument to illustrate the intuition behind our main result and to show how such a problem naturally emerges from applications to collective dynamics.  In particular, in Note \ref{note:comments on main result} we will make several remarks on our main result, provide further motivation, and we will comment  on the difference between the case in which the fast (frozen) process has multiple invariant measures and the case in which it is metastable; this is an important point from a modelling and applications perspective. In Section \ref{subsec:relation to literature} we explain relation to existing literature. Section \ref{sec:non-coupled} and Section \ref{sec:fully-coupled regime} are devoted to proofs (and the organisation of the proofs is better explained in the next section). To highlight main ideas, we have strived for `hands on' proofs, which don't make use of any abstract results.

\section{Heuristics and Motivation}\label{subsec:Heuristics}
We start with a toy example,  which is a simplified version of the collective navigation model which motivated us to look into this problem. To explain the way in which the limit is obtained in the simplest possible scenario, in this example the slow process has a specific drift, the fast process is decoupled from the slow process, and it has two absorbing states (note that absorbing states are just ergodic classes of cardinality one).  We then give a heuristic explanation of how our main result is obtained in two steps:  in step one, we consider the case in which the slow component is a general ODE,  the fast component is decoupled from the slow component and it has $L$ absorbing classes. Then, in the second step, we first generalise  to a fully coupled slow-fast system where the fast (frozen) process has $L$ absorbing states and then to a fast (frozen) process with $L$ ergodic classes. 

 These two steps also broadly mirror the approach of proof we take in Section \ref{sec:non-coupled} and Section \ref{sec:fully-coupled regime}, we will be more precise below. 
 
 {\em  {Toy Example}.}  
 Let $(\slow, \fastuncoupled)$ be a process in $\R^N \times \{-1,1\}^N$, i.e. here $\statespace = \{-1, 1\}^N$, \footnote{As customary, $\{-1,1\}^N$ is the set of vectors of length $N$, whose components are either 1 or -1. } where $\slow$ can be thought of as the vector of positions  at time $t$ of $N$ one-dimensional particles, each of which evolves with velocity given by the $i$-th component of 
$\fastuncoupled$;  that is 
\begin{align}\label{eqn:toyslow}
d\slow = \fastuncoupled \,dt \, , \quad X_0^{\lambda, \bx,\bv}=\bx, V_0^{\lambda, \bv}=\bv\,.
\end{align}
The dynamics of the velocity vector $\fastuncoupled$ evolves  independently of the slow component, according to the following rule: each particle carries a Poisson clock with rate $\lambda$ (the clocks are independent but the rate is the same for all particles); when the $i$-th clock goes off the velocity of particle $i$ switches (if it was $+1$ it becomes $-1$ or viceversa) with a certain probability.  When all the particles reach consensus, i.e. if the process $\fastuncoupled$ hits either one of the states $\pm \absorbing$ where $\absorbing$ is the vector of length $N$ with all entries equal to one, $\mathbf{e} \coloneqq (+1, \ldots, +1)$, then the velocity switches stop and all particles will indefinitely keep travelling with the consensus velocity, $\pm 1$.  That is, if for any $\mathbf{v} = (v_1, \ldots, v_N) \in \{-1, 1\}^N$ we use the notation $\mathbf{v}_{(i)}\coloneqq (v_1, \ldots, -v_i, \ldots ,v_N)$ to denote the vector $\bv$ with the $i^{th}$ component switched,   then the process $\fastuncoupled$ is a CTMP with overall Poisson clock of intensity $N\lambda$ and transition matrix $\cP$ with the following structure: for any $\bv \neq \pm \absorbing$, we have 
\begin{equation}\label{ass:pFunc}
   \transProbs{}(\mathbf{v}, \mathbf{v}') = \begin{cases}
   >0
   & \text{if $\mathbf{v}' = \mathbf{v}_{(i)}$ for some $1\leq i\leq N$,} \\
    0 & \text{otherwise,}
    \end{cases}
\end{equation}
and moreover 
\begin{equation}\label{eqn:toyexample absorbing}
\transProbs{}(\absorbing, \absorbing) = \transProbs{}(-\absorbing, -\absorbing) = 1 \, ,
\end{equation}
so that $\absorbing$ and $-\absorbing$ are (the only) absorbing states.

By standard Markov Chain theory (see e.g. \cite[Theorem 1.3.2]{norris1998markov}), the structure of the transition matrix $\cP$ implies the following two facts: given any initial velocity configuration $\bv$, with probability one the fast process will be  eventually absorbed in either  $\absorbing$ or $-\absorbing$ (the particles will `reach consensus' and eventually will start moving either all with velocity 1 or all with velocity $-1$); for any initial velocity  $\bv$, there is a positive (and computable) probability,  $0\leq q(\bv)\leq 1$, that the process $\fastuncoupled$ is absorbed in $-\absorbing$, so that $0\leq 1-q(\bv) \leq 1$ is the probability that the process is absorbed in $ \absorbing$. We emphasize that this probability depends on the initial state $\bv$ (but not on $\lambda$).  In particular  in this case the state space $\statespace$ is decomposed into two ergodic classes, 
$E^{(1)}= \{\absorbing\}, E^{(2)} =\{ - \absorbing\} $, and all other states are transient. 

In this setting it is relatively easy to guess what the limiting dynamics will look like: as $\lambda$ gets larger, the velocity switches will happen faster and faster so, in the limit $\lambda\rightarrow\infty$, the fast process will reach consensus  `immediately', i.e. it will  be `immediately' absorbed either in $-\absorbing$ or in $\absorbing$, with probability $q(\bv)$ and $1-q(\bv)$, respectively. Therefore, one expects that the limiting dynamics will be described by the following process:
\begin{equation}\label{eqn:avgedProcess}
    d\bar{X}^{\mathbf{x}, \mathbf{v}}_t = \zeta^{\mathbf{v}}  dt, \quad \bar{X}^{\mathbf{x},\mathbf{v}}_0 = \mathbf{x},
\end{equation}
where $\zeta^{\mathbf{v}}$ is a realization of a Bernoulli random variable with 
\begin{equation*}
    \zeta^{\mathbf{v}} = \begin{cases}
    -\absorbing  & \text{with probability } q(\mathbf{v}) \\
    \absorbing & \text{with probability } 1-q(\mathbf{v}).
    \end{cases}
\end{equation*}
In words, the limiting dynamics is as follows:  we flip a coin which has probability $q(\bv)$ of landing on $-\absorbing$ and probability $1-q(\bv)$ of landing on $\absorbing$. If the coin lands on $\absorbing$ then the position vector will evolve according to the ODE with drift $\absorbing$ (i.e. all particles will move with velocity one in a straight line), otherwise according to the ODE with drift $-\absorbing$.  
This result is made precise in Proposition \ref{thm:averaging}, and proved in Section \ref{sec:non-coupled}. Strictly speaking Proposition \ref{thm:averaging} is a corollary of Theorem \ref{thm:averaging1Erg} but we prove it separately, as the simpler setting allows to highlight the important ingredients of the proof. 

\begin{Note}
\textup{ The process $\fastuncoupled$ of our toy example has multiple invariant measures; indeed, for any $\alpha \in [0,1]$,  all  functions $\mu$ on $\{-1, 1\}^N$ of the form
\begin{equation*}
    \mu(\mathbf{v}') = \begin{cases}
    \alpha & \text{if } \mathbf{v'} = -\absorbing, \\
    1-\alpha & \text{if } \mathbf{v}' = \absorbing, \\
     0 & {\text{for any other } \bv' \in \{-1, 1\}^N }\, 
    \end{cases}
\end{equation*}
are invariant measures for the dynamics - and, since all the states except the absorbing states are transient,  these measures are the only ones that matter for the long time behaviour of the dynamics $\fastuncoupled$.  
In particular, the measures $\mu^{(\absorbing)}$ and $\mu^{(-\absorbing)}$,  corresponding to the choice $\alpha=0$ and $\alpha =1$, respectively, are invariant,  and the same holds for  each of the measures $\mu^{\mathbf{v}}_{\infty}$ defined as
\begin{equation}\label{eqn:limMeas}
    \mu^{\mathbf{v}}_{\infty}(\mathbf{v}') = \begin{cases}
    q(\mathbf{v}) & \text{if } \mathbf{v'} = -\absorbing, \\
    (1-q(\mathbf{v})) & \text{if } \mathbf{v}' = \absorbing, \\
     0 & \text{for any other } \bv' \in \{-1, 1\}^N \,.
    \end{cases}
\end{equation}
Note that the above is a family of probability measures on $\{-1,1\}^N$,  indexed by the initial datum $\mathbf{v}$. It is easy to show (see Lemma \ref{lemma:fastProcessConv}) that if $V_0^{\lambda, \mathbf{v}}=\mathbf{v}$ then the law of $\fastuncoupled$ converges to $\mu^{\mathbf{v}}_{\infty}$. In other words, given the initial datum, we know the corresponding limiting invariant measure.
 }
\end{Note}

{\em Step 1. The uncoupled regime.}  The only structure that matters  in our toy example is the fact that the fast process is a CTMP (on any finite state space $\statespace$) with multiple  absorbing states,  $\absorbing^{(1)}, \dots , \absorbing^{(L)}$, and that, for each initial value $\bv$, with probability one,  one of these absorbing states is reached by the dynamics; let $q^{(i)}(\bv)$ be the probability that $\fastuncoupled$ is absorbed in $\absorbing^{(i)}$ (if $V_0^{\lambda,\bv}=\bv$).   It is now easy to see that, under these assumptions on the fast process (and still assuming that $\fastuncoupled$ is independent of $\slow$), if the slow process evolves according to 
\eqref{slowergodic}, 
as $\lambda \rightarrow \infty$, $\slow$ will converge to the solution of 
\begin{equation}\label{limita}
d\slowbar = a(\slowbar, \zeta^{\mathbf{v}}) dt \, \quad \bar X_0^{\bx, \bv} = \bx \, 
\end{equation}
where now $\zeta^{\mathbf{v}}$ is a discrete random variable which can take $L$ values and is equal to $\absorbing^{(i)}$ with probability $q^{(i)}(\bv)$. 
In words, the limiting dynamics is as follows: we have $L$ drifts at our disposal to choose from, namely $a^{(1)}(\bz):=a(\bz, \absorbing^{(1)}), \dots, a^{(L)}(\bz):=a(\bz, \absorbing^{(L)})$. We roll a dice (with $L$ faces): if the dice lands on the $i$-th face then the ODE we run with is the one with drift $a^{(i)}$. To reconcile with \eqref{eqn: informalstatementlimit}, it is worth noting that
\begin{equation}\label{defmui}
a^{(i)}(\bz):=a(\bz, \absorbing^{(i)}) = \sum_{\bv \in \statespace} a(\bz, \bv) \mu^{(i)}(\bv) \, ,
\end{equation}
where $\mu^{(i)}$ is the measure on $\statespace$ defined as $\mu^{(i)}(\absorbing^{(i)}) = 1$ and $\mu^{(i)}(\bv) = 0$ if $\bv \neq \absorbing^{(i)}$ (clearly all the measures $\mu^{(i)}$ are invariant for the fast process, similarly to the measures $\mu^{(\absorbing)}$ and $\mu^{(-\absorbing)}$ for our toy example). This result is a consequence of Theorem \ref{thm:averaging1Erg} so we don't prove it explicitly, but we have presented it here as an intermediate step for explanatory purposes. 

 {\em Step 2. The fully coupled regime.} We now want to understand the limiting dynamics when the fast process is coupled to the slow process. So, assume the slow process evolves according to  \eqref{eqn:fullycoupled}, 
and  $\fastcoupled$ is a $(\tau_t^{\lambda}, \cP_{\slow})$ CTMP. In this case one can use a strategy which combines the intuition given by our toy example with the standard averaging approach for coupled processes: we consider the `frozen process',  obtained from the fast process by freezing the state of the slow component, say to $\slow = \bx$. This frozen process, denoted by $\frozenintro$,  is then a $(\tau_t^{\lambda}, \cP_{\bx})$--CTMP.  Assuming that $\frozenintro$ has $L$ absorbing states (which are the same, irrespective of the value of $\bx$ we fix), and that, given any initial datum $\tilde{V}_0^{\lambda, \bx, \bv} = \bv$,  there is probability one of reaching  one of them, one can  consider the probability $q^{(i)}(\bx, \bv)$ of $\frozenintro$ being absorbed in $\absorbing^{(i)}$. The limiting  dynamics is then given by \eqref{limita}, where now $\zeta^{\mathbf{v}}$ is replaced by $\zeta^{\mathbf{x},\mathbf{v}}$, where $\zeta^{\mathbf{x}, \bv}$ is equal to $\absorbing^{(i)}$ with probability $q^{(i)}(\bx, \bv)$. This fact is proven in Theorem \ref{thm:averaging1}. Once this result is proved, since absorbing states are just ergodic classes of cardinality one,  the main  theorem follows as a straightforward extension of Theorem \ref{thm:averaging1}, see Section \ref{sec:ergClasses}.

\begin{Note}\label{note:comments on main result} \textup{ Let us make some comments on our main result, which we informally stated in Section \ref{subsec:informalstatement}.
\\
  $\bullet$ There are at least  two main differences between doing averaging when the fast (frozen) process is ergodic and when it is not: i) whether the   slow-fast processes are coupled or not, in the non-ergodic setting the limiting dynamics \eqref{eqn: informalstatementlimit}  depends on the initial datum $\bv$ of the fast dynamics. This is not true when the fast process is ergodic as, in view of ergodicity, the fast process would forget  its initial state and converge to the invariant measure, irrespective of its initial datum $\mathbf{v}$; 
ii) in the non-ergodic case the limiting dynamics is random, even if the slow motion is governed by an ODE. This is in contrast with the case in which the fast process is ergodic.  Indeed, in this case, if the slow component is an ODE,  the limiting {\em averaged}   dynamics is always deterministic, see \eqref{eqn:averaged} and \eqref{intro:averagedcoupled}. For the sake of completeness let us  point out that in the homogenization regime (see \cite[Chapters $11$ and $18$]{pavliotis2008multiscale}) the limiting dynamics can be stochastic even if the slow variable is described by an ODE; however stochasticity in the homogenization setting comes from the fact that the limit can be a stochastic differential equation (SDE), that is, an ODE perturbed by random noise. The limiting dynamics \eqref{eqn: informalstatementlimit} is random in a different sense, namely in  the sense that the coefficient of the averaged dynamics is random, so \eqref{eqn: informalstatementlimit} is a random ODE.   \\
$\bullet$ There is a natural question concerning comparison with the case in which the fast process is metastable (more precisely, the case in which the frozen process is metastable, for every value of $\bx$). We are thinking for example of the Langevin equation in double well potential, see e.g. \cite{pavliotis2014stochastic}.  The modelling of the effects of the fast process is a very delicate matter, see e.g. \cite{weinan2011principles} and we don't discuss it here. However we point out that metastable processes are processes with {\em one} invariant measure, but such a measure is (highly) multimodal. Because of uniqueness of the invariant measure, if the fast process is metastable the averaging regime can  be tackled with known techniques, and the results of this paper are not needed. In particular, if the slow variable is described by an ODE,  the averaged process is \eqref{intro:averagedcoupled}, and it is deterministic. In the many-invariant measures regime of this paper the averaged process is given by \eqref{eqn: informalstatementlimit}, and it is random.\\
$\bullet$ In this work we assume that the frozen process will eventually be absorbed by/reach an ergodic class  (Assumption \ref{ass:pErg} \eqref{ass:pMinorErg} in  Section \ref{sec:ergClasses}) - which, from the point of view of migration, can be interpreted either as the group eventually reaching one of many possible targets or, in the case of two targets, as successful/failed migration. Moreover, we assume that the frozen process has the same (number of) ergodic classes, irrespective of $\bx$ (Assumption \ref{ass:pErg} \eqref{ass:p:AbsErg} in  Section \ref{sec:ergClasses}). This assumption will hold or not depending on the model at hand.  In particular, this assumption is satisfied  in the applications we have  in mind (on which we will comment further below) and  in the queueing systems studied in \cite{switchingDiffusionsPaper},  but it does not hold in other scenarios. Consider for example McKean-Vlasov type dynamics, which are non-linear SDEs obtained in the large particle limit of interacting particle systems, see e.g. \cite{sznitman1991topics}.  Depending on the choice of drift and on   the strength of the  noise coefficient, these equations may exhibit a phase transition \cite{dawson1983critical, herrmann2010non, carrillo2020long, angeli2023well}: typically, for large values of the noise the dynamics has only one invariant measure but for small values of the noise the SDE can have multiple invariant measures. If a McKean-Vlasov equation of this type was used to model the fast process (and the slow variable appeared in the noise coefficients of the fast process, scenario which is at the moment work in progress from the second author and collaborators) then clearly the assumption that we make in this paper would not hold. However, still within the class of McKean-Vlasov dynamics, evolutions inspired by the Vicsek model have been shown to have the same number of invariant measures, irrespective of the strength of the noise - phenomenon which goes under the name of {\em unconditional flocking}, see \cite{butta2022non, butta2018non,garnier2016mean, bertini2014synchronization}. In this case the assumption we make in this paper would be satisfied.  \\
$\bullet$ The proof of our main result is constructive. The main trick there is to introduce not only  the frozen process $\fastFrozen$, which makes use of the same transition matrix $\cP_{\bx}$ every time the Poisson clock goes off, but also a further intermediate process, $V^{\lambda, \bx, \bv, \ell}_t$, where $\ell=(\bx^{(j)})_{j=1}^{\infty}$ is a fixed sequence of points in $\R^N$. Such a process moves according to the transition probability $\cP_{\bx^{(j)}}$ when the Poisson clock goes off for the $j$-th time: intuitively, if the sequence $(\bx^{(j)})_{j=1}^{\infty}$ was a discretization of the slow dynamics, then $V^{\lambda, \bx, \bv, \ell}_t$ would remain close to  $\fastcoupled$, see Section \ref{sec:fully-coupled regime} for details. 
  }  
\end{Note}

Let us now further comment on motivation for this work and, more broadly, for studying multiscale dynamics when the fast process has multiple invariant measures.  The work of this paper was first inspired  by the collective navigation model introduced in \cite{painter}. In \cite{painter}, a piecewise deterministic Markov process is used to model the position and velocity of each individual in a group of animals. At random  (Poisson) times  an animal makes a decision on its heading based on the information available to it, and the authors distinguish between inherent information (obtained e.g. through the environment) and the collective information (obtained through the headings of the other animals in the group). Each type of information can be acquired at different frequency and one can imagine that one of them is obtained much faster than the other, creating a scale separation.  The toy model we have presented can be seen as  a simplified version of the model \cite{painter}, where we focus on just one type of information being acquired (the "fast one"). 

Other collective navigation models where  the study of multiscale systems in the non-ergodic setting is relevant are e.g. the recent \cite{couzinGeometrical, couzin2} where a model for group navigation   of animals towards one of two targets is introduced - again,  a spin model of the velocities was used, in conjunction with experimental data of fish moving towards targets.    In this model the animals initially move towards the midpoint of the targets (that is, they are initially in the "compromise region"); after a certain  point in space they transition from moving towards the midpoint of the targets  to deciding which of the two targets they will  move towards (this region is referred to as the ``decision'' region) and this can be viewed as a changing of the structure of the velocity process,  from having one invariant measure to two. If one starts the animals in the ``decision'' region then a similar effect can be attained using our toy model. If, however, the group starts in the ``compromise'' region, then the work of \cite{friedlin} is more relevant, as the particles move from a region with one invariant measure (taking the average heading of both targets) to two invariant measures (moving towards only one of the two targets). In either case, the models and data in \cite{couzinGeometrical, couzin2} and the references therein are indicative of the fast dynamics (in this case a process describing the velocities of all the particles) exhibiting more complicated invariant measure structure than universal uniqueness. 

Aside from these specific models, the study of multiscale systems when the fast process has multiple invariant measures naturally emerges when considering multiscale interacting particle systems, see e.g. \cite{delgadino, barre2021fast, gomes2018mean} to mention just a few references on the matter, which have recently attracted a lot of attention in stochastic analysis, motivated by applications to the study of multiagent systems in social dynamics, animal migration, opinion formation, deep learning etc. (note that the models we consider here can be seen as multiscale interacting particle systems). In this context one often  needs to consider two limits, the large population size limit ($N\rightarrow \infty$, where $N$ is the number of particles in the population) and the scale separation limit ($\epsilon \rightarrow 0$, to use the notation of the cited references, $\epsilon =1/\lambda$ to compare with our notation). Typically, when $N$ is fixed the fast process has a unique invariant measure, so one knows how to take the limit $\epsilon \rightarrow 0$ first and $N\rightarrow \infty$ afterwards. However, if one lets $N\rightarrow \infty$ first, the resulting dynamics can be of McKean-Vlasov type and exhibit multiple invariant measures. While it is known that in general these two limits do not commute, understanding the second limit ($N\rightarrow \infty$ first) is still  largely and open problem, which the second  author of this paper and collaborators are working on.\\

\subsection{Relation to the literature} \label{subsec:relation to literature}
The literature on averaging 
(for SDEs, stochastic partial differential equations (SPDEs), jump processes etc) in the case in which the fast (frozen) process has a unique invariant measure  is  vast and by now well-developed. Hence, a complete review is outside of the scope of this paper, and we only mention the excellent review books \cite{pavliotis2008multiscale, weinan2011principles} and some recent results, which include:  the case when  the fast process is in compact  \cite{xuemei} or non-compact state space \cite{pardoux2001, pardoux2003},  time dependent  \cite{liurocknerTime, UDA20211},  non-linear (in the sense of McKean) \cite{spiliopolus}.  A first result which proves convergence to the limiting dynamics to hold uniformly in time has also been produced \cite{crisan2022poisson}; the non-Markovian setting \cite{xue-meiFrac} and  The case in which the fast process is periodic have been investigated as well, see e.g. \cite{UDA20211} and references therein.  In comparison, the study of the case in which the fast process is non-ergodic is at early stages and there are, to the best of our knowledge, very few works on the matter. In particular,  while completing this paper we became aware of the works \cite{friedlin},\cite{switchingDiffusionsPaper} which seem to be the only other two works  who are close in spirit and context  to the present paper. The specific setup of \cite{friedlin}, \cite{switchingDiffusionsPaper}  and the methods of proof used there are different from ours, but there are similarities, so we come to compare.   In \cite{friedlin} the authors consider slow-fast systems in which the slow variable evolves in one dimension (in $\R$) according to either an ODE or an SDE,  and the fast process is a Markov Chain on finite state space $\statespace$. When the slow process evolves according to an ODE, say $dX_t = a(X_t, V_t) dt$,  they make the following assumption: the function $a$ is always strictly positive (so that $X_t$ is increasing) and when $x<0$ the fast process $V_t$ has only one invariant measure, say $\mu^{(0)}$, while for $x\geq 0$ it has two ergodic classes, and, correspondingly, two invariant measures, $\mu^{(1)}$ and $\mu^{(2)}$. The limiting dynamics is then as follows: assuming $X_0=x<0$, the dynamics evolves according to 
 $$
 d\bar X_t = \bar a^{(0)} (\bar X_t) dt, \quad \bar X_0=x , \quad \bar a^{(0)}(\bx): = \sum_{\mathbf v \in \statespace} a(\bx,\mathbf v) d\mu^{(0)}(\mathbf v) ,  
$$
until it hits the origin $x=0$; at that point we flip a coin and with probability $q^{(i)}$, $i \in \{1,2\}$, for which an expression is provided,   the process  $\bar X_t$ continues according to 
$$
 d\bar X_t = \bar a^{(i)} (\bar X_t) dt, \quad \bar X_0=x , \quad a^{(i)}(\bx): = \sum_{\bv \in E^{(i)}} a(\bx,\bv) d\mu^{(i)}(\bv) \,.
$$
In  \cite{friedlin}, like in the present paper, once the slow variable becomes positive, the fast process `picks' an ergodic class and the slow process is such that the fast process will remain in that ergodic set.  \\
There are clear similarities between the setting of \cite{friedlin} and the setting of this paper, but also key differences. Firstly, our setting is `dimension independent', in the sense that the slow process can live in any dimension, not just on the real line. 
 In \cite{friedlin} the `flip of the coin' happens when the slow process hits the origin and the probabilities $q^{(i)}$ of being absorbed in either ergodic class do not depend on the initial datum of the fast dynamics (or on the initial datum of the slow dynamics either). Thus, the limiting dynamics does not depend on the initial condition of the fast process either. This is because in the setting of \cite{friedlin} the fast process has no transient states. As a consequence  of this, it is natural in their case to think of the limiting process as a process on a graph, with the graph as in \cite[Figure $1$]{friedlin}. In our case we could think of the process as living on a family of graphs, indexed by the initial conditions $(\bx, \bv)$ (or simply $\bv$ in the uncoupled case). However, in our setup we find it simpler to think of it as a random ODE. 

 In \cite{switchingDiffusionsPaper} the slow variable is a jump-diffusion process, where the jump part can force the fast process to move between ergodic classes and the fast process evolves on a finite state space. The number of ergodic classes is fixed, irrespective of $\bx$, and there are no transient states.  
 
 Another related stream of literature is the one concerned with  singularly perturbed ODEs, see for example \cite{artstein_vigodner_1996, Artstein2002ONSP,Arnold,dumortier1996canard,KrupaCanard}, and references therein.  We also flag up \cite{freidlin2004diffusion}, which tackles the issue of multiple invariant measures for the fast dynamics in the context of multiscale Partial differential equations.
 
Finally, we mention that the processes considered in this paper can be seen as systems of interacting Piecewise Deterministic Markov Processes 
(PDMPs, also referred to as velocity-jump processes or run-and-tumble); see \cite{davis1984piecewise}. PDMPs  are very popular  both as modelling tools in mathematical biology,  see \cite{lawley}, and as sampling  mechanisms  \cite{BERTAZZI202291, roberts}.

\section{Non fully-coupled regime: toy example}\label{sec:non-coupled}

Throughout this section, let $\slow, \fastuncoupled$ be the process of the toy example introduced in  Section \ref{subsec:Heuristics}, see \eqref{ass:pFunc}-\eqref{eqn:toyexample absorbing}, and let $\slowbar$ be as in  \eqref{eqn:avgedProcess}. Then the following holds.

\begin{prop}\label{thm:averaging} 
For any $t>0$ there exists a constants $K>0$ (possibly dependent on $t$, but independent of $\lambda$) and a constant $c_1>0$ such that 
		\begin{equation}\label{eqn:conclnonfullycoupled}
			\left|\mathbb{E}\big[f\big(X^{\lambda,\mathbf{x}, \mathbf{v}}_t \big) \big]-  \mathbb{E}\left[f(\bar{X}^{\mathbf{x},\mathbf{v}}_t)\right]\right| \leq K\|f\|\left( \frac{1}{\sqrt{\lambda}}+e^{-c_1 \sqrt{\lambda}}\right)  \, ,  
		\end{equation}
		for all $\mathbf{x} \in \R^N, \mathbf{v} \in \{-1,+1\}^N$ and $f\colon \R^N \rightarrow \mathbb{R}$ such that $f \in C^1_b(\R^N);$\footnote{This is the space of continuous and bounded real-valued functions on $\R^N$ with continuous and bounded derivative.} in the above $\|f\| \coloneqq \|f\|_\infty + \sum_{i} \|\partial_{x_i} f\|_{\infty}$ and $\| \cdot \|_{\infty}$ is the supremum norm. 
\end{prop}
The rate of convergence in \eqref{eqn:conclnonfullycoupled} is not optimal, see footnote \ref{nonoptimal}.
As already commented, by \eqref{ass:pFunc} with probability one the fast process  eventually hits one of the absorbing states. That is,
\begin{equation} \label{eqn:convAbs}
    \mathbb{P}[V^{\lambda, \mathbf{v}}_t \in \{\absorbing,-\absorbing\}] \rightarrow 1 , \quad \text{as } t \rightarrow\infty,
\end{equation}
for all $\lambda>0, \mathbf{v}\in \{-1,+1\}^N$.  One can obtain Proposition \ref{thm:averaging} assuming  \eqref{eqn:convAbs} only,   i.e one can use \eqref{eqn:convAbs} instead of \eqref{ass:pFunc}, but here we stick with the more concrete assumption \eqref{ass:pFunc}. Note also that 
we can write down  expectations  of the process \eqref{eqn:avgedProcess} explicitly: for any $f\colon \R^N \rightarrow \mathbb{R}$ we have
\begin{equation}\label{eqn:avgedSG}
    \mathbb{E}\left[f(\bar{X}^{\mathbf{x},\mathbf{v}}_t)\right] = q(\mathbf{v})f(\mathbf{x}-t\absorbing) + (1-q(\mathbf{v}))f(\mathbf{x}+t\absorbing),
\end{equation}
 and this fact will be explicitly used in the proof. Before proving the above proposition, we show in the next lemma that the  fast dynamics $V^{\lambda, \mathbf{v}}_t$ (defined after \eqref{eqn:toyslow}) converges to the limiting measure \eqref{eqn:limMeas}. This lemma would be a trivial statement if the fast process had only one invariant measure (as $\fastuncoupled$ is just a Markov Process on finite state space); even with the process having many invariant measures the proof is still relatively simple, but we could not find it in the literature so for completeness we include the proof of Lemma \ref{lemma:fastProcessConv} in the appendix. 
\begin{lemma}\label{lemma:fastProcessConv}
Let \eqref{ass:pFunc}-\eqref{eqn:toyexample absorbing} hold, and let $\mu_t^{\lambda, \bv}$ denote the law of $\fastuncoupled$. Then the following holds: there exist $K_1,c_1>0$ (independent of $\mathbf{v} \in \{-1,+1\}$, $\lambda>0$ and $t>0$) such that 
\begin{align}\label{eqn:convAbs0}
&|\mu_{t}^{\lambda, \mathbf{v}}(\{\absorbing,-\absorbing\}) - 1| 
\leq K_1\rate{t\lambda} \\
\label{eqn:convAbs1}
    &|\mu_{t}^{\lambda, \mathbf{v}}(-\absorbing) - q(\mathbf{v})|
    \leq K_1\rate{t\lambda}\\
    \label{eqn:convAbs2}
        &|\mu_{t}^{\lambda, \mathbf{v}}(\absorbing) - (1-q(\mathbf{v}))|
        \leq K_1\rate{t\lambda}.
\end{align}
Hence, there exists $K>0$ such that 
\begin{equation}\label{eqn:convMeas}
    \| \mu_{t}^{\lambda, \mathbf{v}} - \mu^{\mathbf{v}}_{\infty} \|_{\text{TV}} \leq K\rate{t\lambda},
\end{equation}
where $\|\cdot\|_{TV}$ is the total variation norm, i.e $\|\mu_{t}^{\lambda, \mathbf{v}}\|_{TV} = \sum_{\bv' \in \fastDomain}|\mu_{t}^{\lambda, \mathbf{v}}(\bv')|$, and $\mu^{\mathbf{v}}_{\infty}$ has been introduced in \eqref{eqn:limMeas}.
\end{lemma}
\begin{proof}
The proof of Lemma \ref{lemma:fastProcessConv} can be found in Appendix \ref{sec:AuxProofs}
    
\end{proof}

\begin{proof}[Proof of Proposition \ref{thm:averaging}]
We begin by defining $t_0 \coloneqq \frac{1}{\sqrt{\lambda}}$.\footnote{\label{nonoptimal} We note that this can be pushed to $\lambda^{-l}$ for any $l<1$, and that in this case, we recover the rate $\lambda^{-l}$ on the right hand side of \eqref{eqn:conclnonfullycoupled}. To see this, observe that in the following we only need that $t_0\lambda \rightarrow\infty$ as $\lambda\rightarrow\infty$. For simplicity we keep $t_0 = \frac{1}{\sqrt{\lambda}}$.} Using the properties of conditional expectation, and letting $t>t_0$ we can write the following
\begin{equation}\label{eqn:conditioning}
    \begin{split}
            \mathbb{E}\big[f\big(X^{\lambda,\mathbf{x}, \mathbf{v}}_t \big) \big]
            &= \underbrace{\mathbb{P}[ V^{\lambda, \mathbf{v}}_{t_0} =-\absorbing]\mathbb{E}\big[f\big(X^{\lambda,\mathbf{x}, \mathbf{v}}_t \big) | V^{\lambda, \mathbf{v}}_{t_0} = -\absorbing\big]}_{\eqqcolon A} \\
            &+ \underbrace{\mathbb{P}[ V^{\lambda, \mathbf{v}}_{t_0} = \absorbing]\mathbb{E}\big[f\big(X^{\lambda,\mathbf{x}, \mathbf{v}}_t \big) | V^{\lambda, \mathbf{v}}_{t_0} = \absorbing\big]}_{\eqqcolon B} \\
            &+ \underbrace{\mathbb{P}[ V^{\lambda, \mathbf{v}}_{t_0} \notin \{\absorbing, -\absorbing\}]\mathbb{E}\big[f\big(X^{\lambda,\mathbf{x}, \mathbf{v}}_t \big) | V^{\lambda, \mathbf{v}}_{t_0} \notin \{\absorbing, -\absorbing\}\big]}_{\eqqcolon C}
    \end{split}
\end{equation}
From \eqref{eqn:avgedSG} and \eqref{eqn:conditioning} we have
\begin{equation}\label{eqn:decomp}
\begin{split}
    &\left|\mathbb{E}\big[f\big(X^{\lambda,\mathbf{x}, \mathbf{v}}_t \big) \big]-  \mathbb{E}\left[f(\bar{X}^{\mathbf{x},\mathbf{v}}_t)\right]\right|\\ &\leq |A - q(\mathbf{v})f(\mathbf{x}-t\absorbing)| +  |B - (1-q(\mathbf{v}))f(\mathbf{x}+t\absorbing)| + |C|.
\end{split}
\end{equation}

Taking into account \eqref{eqn:decomp}, \eqref{eqn:conclnonfullycoupled} follows if we show there exists a constant $K>0$ such that the following three inequalities hold
\begin{enumerate}
\item
\begin{equation}\label{eqn:1}
    |A - q(\mathbf{v})f(\mathbf{x}-t\absorbing)| \leq K \|f\|\left(\frac{1}{\sqrt{\lambda}}+\rate{\sqrt{\lambda}}\right),
\end{equation}
\item
\begin{equation}\label{eqn:2}
    |B - (1-q(\mathbf{v}))f(\mathbf{x}+t\absorbing)| \leq K\|f\|\left(\frac{1}{\sqrt{\lambda}}+\rate{\sqrt{\lambda}}\right).
\end{equation}
\item Finally,
\begin{equation}\label{eqn:3}
      |C| \leq K\|f\|\rate{\sqrt{\lambda}}.
\end{equation}
\end{enumerate}
By using the fact that for any three random variables $X,Y,Z$, the following holds $\E[\E[X|Y,Z]|Z] = \E[X|Z]$, we can rewrite $A$ as
\begin{equation}\label{eqn:towerProp}\begin{split}
    A &= \mathbb{P}[ V^{\lambda, \mathbf{v}}_{t_0} =-\absorbing]\mathbb{E}\big[f\big(X^{\lambda,\mathbf{x}, \mathbf{v}}_t \big) | V^{\lambda, \mathbf{v}}_{t_0} = -\absorbing\big] \\
    &=\mathbb{P}[ V^{\lambda, \mathbf{v}}_{t_0} =-\absorbing]\mathbb{E}\big[\mathbb{E}\big[f\big(X^{\lambda,\mathbf{x}, \mathbf{v}}_t \big)| V^{\lambda, \mathbf{v}}_{t_0} = -\absorbing, X^{\lambda,\mathbf{x}, \mathbf{v}}_{t_0} \big]| V^{\lambda, \mathbf{v}}_{t_0} = -\absorbing \big].
\end{split}
\end{equation}
Now we use the fact that if $V^{\lambda, \mathbf{v}}_{t_0} = -\absorbing$, then
the dynamics of $X^{\lambda,\mathbf{x}, \mathbf{v}}_{t}$ is deterministic for $t\geq t_0$:
\begin{equation}\label{eqn:condBoth}
    \mathbb{E}\big[f\big(X^{\lambda,\mathbf{x}, \mathbf{v}}_t \big)| V^{\lambda, \mathbf{v}}_{t_0} = -\absorbing, X^{\lambda,\mathbf{x}, \mathbf{v}}_{t_0} \big] = f(X^{\lambda,\mathbf{x}, \mathbf{v}}_{t_0} - (t-t_0)\absorbing).
\end{equation}
Using \eqref{eqn:towerProp} and \eqref{eqn:condBoth}, we obtain
\begin{equation}\label{eqn:A}
    A = \mathbb{P}[ V^{\lambda, \mathbf{v}}_{t_0} =-\absorbing]\mathbb{E}\big[f(X^{\lambda,\mathbf{x}, \mathbf{v}}_{t_0} - (t-t_0)\absorbing)| V^{\lambda, \mathbf{v}}_{t_0} = -\absorbing].
\end{equation}
By \eqref{eqn:A} and the triangle inequality, we have
\begin{equation}\label{eqn:AqN}
\begin{split}
     &|A - q(\mathbf{v})f(\mathbf{x}-t\absorbing)| \\
     &\leq \underbrace{\mathbb{P}[ V^{\lambda, \mathbf{v}}_{t_0} =-\absorbing]|\mathbb{E}\big[f(X^{\lambda,\mathbf{x}, \mathbf{v}}_{t_0} - (t-t_0)\absorbing) | V^{\lambda, \mathbf{v}}_{t_0} = -\absorbing\big]-f(\mathbf{x}-t\absorbing)|}_{\eqqcolon D} \\ &+  |f(\mathbf{x}-t\absorbing)||\mathbb{P}[ V^{\lambda, \mathbf{v}}_{t_0} =-\absorbing] - q(\mathbf{v})|.
\end{split}
\end{equation}
Since $f$ is Lipschitz continuous, we also have \begin{equation}\label{eqn:D}
\begin{split}
       D &\leq \mathbb{E}\left[\left|f(X^{\lambda,\mathbf{x}, \mathbf{v}}_{t_0} - (t-t_0)\absorbing)-f(\mathbf{x}-t\absorbing)\right| \, \mid V^{\lambda, \mathbf{v}}_{t_0} = -\absorbing\right] \\ &\leq \|f\| \mathbb{E}\left[\|X^{\lambda,\mathbf{x}, \mathbf{v}}_{t_0}-\mathbf{x} + t_0\absorbing  \|_2 \, \mid V^{\lambda, \mathbf{v}}_{t_0} = -\absorbing\right] \\
       &\leq \|f\|\left(\mathbb{E}\left[\|X^{\lambda,\mathbf{x}, \mathbf{v}}_{t_0}-\mathbf{x}\|_2 \, \mid V^{\lambda, \mathbf{v}}_{t_0} = -\absorbing\right] + t_0\|\absorbing  \|_2 \right) \\
       &\leq  K\|f\|\cdot \frac{1}{\sqrt{\lambda}},
\end{split}
\end{equation}where $\|\cdot\|_2$ is the Euclidean norm. In the last inequality we have also used the fact that $X^{\lambda,\mathbf{x}, \mathbf{v}}_{t_0}$ is necessarily in a (Euclidean, closed) ball of radius $t_0\sqrt{N}$ around $\mathbf{x}$, which is due to the fact that each particle has velocity of magnitude $1$.
By \eqref{eqn:convAbs1} we have
\begin{equation}\label{eqn:fLeqEpsilon}
\begin{split}
        |f(\mathbf{x}-t\absorbing)|&|\mathbb{P}[ V^{\lambda, \mathbf{v}}_{t_0} =-\absorbing] - q(\mathbf{v})| \\
        &\leq \|f\|_{\infty}|\mathbb{P}[ V^{\lambda, \mathbf{v}}_{t_0} =-\absorbing] - q(\mathbf{v})| \leq K_1\rate{\sqrt{\lambda}}\|f\|_\infty,
\end{split}
\end{equation}
where $\|\cdot\|_\infty$ is the supremum norm.
From \eqref{eqn:AqN}, \eqref{eqn:D} and \eqref{eqn:fLeqEpsilon}, \eqref{eqn:1} follows.
The inequality \eqref{eqn:2} is shown in exactly the same way, using \eqref{eqn:convAbs2} instead of \eqref{eqn:convAbs1}.

We also have that \begin{equation*}\label{eqn:notInComponent}
    |C| \leq \mathbb{P}[ V^{\lambda, \mathbf{v}}_{t_0} \notin \{\absorbing, -\absorbing\}] \|f\|_{\infty} \leq K_1\rate{\sqrt{\lambda}}\|f\|_{\infty},
\end{equation*}
where we used \eqref{eqn:convAbs0}. Hence \eqref{eqn:3} is shown. Then by \eqref{eqn:decomp}, \eqref{eqn:1}, \eqref{eqn:2} and \eqref{eqn:3}, \eqref{eqn:conclnonfullycoupled} follows so the proof is complete.
\end{proof}

\section{Fully coupled regime}\label{sec:fully-coupled regime}
In this section we consider the setup of Step $2$ in Section \ref{subsec:Heuristics}; that is, $\slowProcess{\lambda,\mathbf{x},\mathbf{v}}{t}$ and $\fast_t$ are fully coupled, and evolve according to $\eqref{eqn:fullycoupled}$.
For ease of proof, as outlined in Step $2$ of Section \ref{subsec:Heuristics}, we split this section into two subsections; in Section \ref{sec:absStates} we start from the case when the frozen process $\fastFrozen_t$ has $\numAbs$ absorbing states $\mathbf{e}^{(1)}, \ldots,  \mathbf{e}^{(\numAbs)} \in \fastDomain$. Then in Section \ref{sec:ergClasses} we state and prove the more general result, Theorem \ref{thm:averaging1Erg}, when $\fastFrozen_t$ has $\numAbs$ ergodic classes.
First, we make the following assumption on $a$ which will hold throughout Section \ref{sec:absStates} and Section \ref{sec:ergClasses}.
\begin{assumption}\label{ass:driftSlow}
Let $a\colon \R^N \times\fastDomain \rightarrow \R^N$.
    For every $\mathbf{v} \in \fastDomain$, the function $a(\cdot, \mathbf{v})$ is continuous and bounded.
\end{assumption}
Note that since $\fastDomain$ is finite, Assumption \ref{ass:driftSlow} means, in particular, that the function $a$ is bounded and continuous, uniformly in $\mathbf{v}$. If  $\fastFrozen_t$ is the frozen process (which, we recall, is a $(\tau^{\lambda}_t, \transProbs{\mathbf{x}})$-CTMP) then $\fastFrozenDisc_k$ denotes the discrete-time version of the frozen process, namely the process that changes value at every time-unit (as opposed to when the Poisson clock goes off) by using the transition probability $\cP_{\bx}$.  In particular,  denoting the law of $\fastFrozen_t$ by $\fastFrozenLaw_t$ and the law of $\fastFrozenDisc_k$ by $\fastFrozenDiscLaw_k$, the following relation between $\fastFrozenLaw_t$ and $\fastFrozenDiscLaw_k$ holds:
\begin{equation}\label{eqn:condPoissonProcess}
    \fastFrozenLaw_t =  \sum_{k=0}^{\infty} \mathbb{P} \left[\tau_t^{\lambda} = k\right] \fastFrozenDiscLaw_k.
\end{equation}
Since  the slow and fast components are coupled, we can't study the fast dynamics independently of the slow one. For this reason we will use  two intermediate processes, the frozen process, $\fastFrozen_t$, and a process which, intuitively,  depends on a fixed sequence of states of the slow variable, $V_t^{\lambda, \bx, \bv, \ell}$, see Section \ref{sec:absStates} for a precise definition. 
\subsection{Frozen process with multiple absorbing states}\label{sec:absStates}
 In this subsection we assume that for every $\mathbf{x} \in \R^N$, the frozen process $\fastFrozen_t$ associated with the fast process $\fastcoupled$ appearing in \eqref{eqn:fullycoupled} has $\numAbs$ absorbing states $\mathbf{e}^{(1)}, \ldots,  \mathbf{e}^{(\numAbs)} \in \fastDomain$. We make the following assumption on the transition probabilities of the fast process.
\begin{assumption}\label{ass:p}
The following hold for the function $\transProbs{}\colon \R^N \times \fastDomain\times\fastDomain \rightarrow[0,1]$.
\begin{enumerate}[(i)]
\item \label{ass:pProb} Transition probabilities: For all $\mathbf{x}\in \R^N$ and $\mathbf{v} \in \fastDomain$, $\sum_{\mathbf{v}' \in \fastDomain} \transProbs{\mathbf{x}}(\mathbf{v},\mathbf{v}')=1.$
        \item \label{ass:p:Abs} The states $\absorbing^{(1)},\ldots, \absorbing^{(\numAbs)} \in \fastDomain$ are absorbing states: For all $1\leq i\leq \numAbs$ and $\mathbf{x}\in \R^N$,
            $\transProbs{\mathbf{x}}(\absorbing^{(i)}, \absorbing^{(i)}) = 1$.
        \item \label{ass:pMinor} Probability of absorption: There exist $\tilde{n}\geq1$ and $z_0<1$ such that for all $\bx \in \R^N, \bv \in \fastDomain$,\begin{equation*}
            \mathbb{P}[\fastFrozenDisc_{\tilde{n}} \notin \{\absorbing^{(1)}, \ldots,\absorbing^{(\numAbs)}\}] \leq z_0 .
        \end{equation*}
        \item \label{ass:pLip}  Lipschitzianity in $\bx$: There exists $K_0>0$ such that for all $\mathbf{v}\in \fastDomain$ and $\mathbf{x},\mathbf{x'} \in \R^N$,
    \begin{equation*}
        \sum_{\mathbf{v'} \in \fastDomain}|\transProbs{\mathbf{x}}( \mathbf{v}, \mathbf{v'}) - \transProbs{\mathbf{x'}} (\mathbf{v}, \mathbf{v'}) | \leq K_0\|\mathbf{x}-\mathbf{x'}\|_1
    \end{equation*}
    where $\|\mathbf{w}\|_1 = \sum^n_{i=1}|w_i|$ for any $\mathbf{w} \in \R^{N}$.
    \end{enumerate}
\end{assumption}
Because the fast process evolves in finite state space, $\eqref{ass:pProb}-\eqref{ass:pLip}$ of Assumption \ref{ass:p} imply the state space decomposition $\statespace = \{E^{(1)} \cup  \dots \cup E^{(L)} \cup \mathcal T \}$, outlined in Section \ref{subsec:informalstatement}, with $E^{(i)} =\absorbing^{(i)}$. Assumption \ref{ass:p} \eqref{ass:pMinor} implies that the fast process will be absorbed into some absorbing state $\absorbing^{(i)}$, and Assumption \ref{ass:p}\eqref{ass:p:Abs}  implies that slow variable will never have the effect of making the fast process change absorbing state. Moreover, Assumption \ref{ass:p} \eqref{ass:pMinor} requires the rate of absorption to be independent of $\bx$. This implies that the constants in the convergence rates of Lemma \ref{lemma41} and Lemma \ref{lemma:fastProcessConvX} are independent of $\bx$.

Let $\LawLimitPosDep$ be the measure on $\statespace$ defined as follows:
\begin{equation}\label{eqn:limMeasX}
    \LawLimitPosDep(\mathbf{v'}) \coloneqq \begin{cases}
    q^i(\mathbf{x}, \mathbf{v}) & \text{if } \mathbf{v'} = \absorbing^{(i)} \text{ for some }1\leq i\leq \numAbs, \\
     0 & \text{otherwise,}
    \end{cases}
\end{equation}
where $q^i(\mathbf{x},\mathbf{v})$ is the probability that $\fastFrozen_t$ is absorbed in $\absorbing^{(i)}$; see, for example, \cite[Chapter 3]{norris1998markov} for details on how to calculate $q^i(\mathbf{x},\mathbf{v})$ using the transition probabilities. 

The next lemma (which is the analogue of Lemma \ref{lemma:fastProcessConv} in the decoupled case) states that $\LawLimitPosDep$ is the limiting distribution of $\fastFrozen_t$.
\begin{lemma} \label{lemma41}
Let Assumption \ref{ass:p} hold. Then there exist $K, c_1>0$, independent of $\mathbf{x} \in \R^N$, $\mathbf{v} \in \fastDomain$, $\lambda>0$ and $t\geq 0$, such that 
\begin{equation}\label{eqn:convMeasXX}
    \| \tilde{\mu}_{t}^{\lambda, \mathbf{x}, \mathbf{v}} - \LawLimitPosDep \|_{\text{TV}} \leq K\rate{t\lambda},
\end{equation}
\end{lemma}
\begin{proof}
    We do not give details of the proof, since the argument is completely analogous to the one in the proof of Lemma \ref{lemma:fastProcessConv}. However, note that $K$ and $c_1$ in \eqref{eqn:convMeasXX} are independent of $\mathbf{x}$ in this case because of Assumption \ref{ass:p} $\eqref{ass:p:Abs}-\eqref{ass:pMinor}$.
\end{proof}
It is useful to introduce a further process $\fastProcessSeqCont_t$, depending on an extra parameter, namely on a sequence $l = \left(\mathbf{x}^{(j)}\right)_{j=1}^{\infty}$ with each $\mathbf{x}^{(j)} \in \R^N$. This process is a continuous time random process with an underlying Poisson clock (of rate $\lambda$), for which the transition matrix at the $k^{th}$ switch is $\transProbs{\mathbf{x}^{(k)}}$ (this process is time-inhomogenous). We denote the law of $\fastProcessSeqCont_t$ by $\fastProcessSeqContLaw_t$. We denote the marginal of the law of the fully coupled dynamics with respect to the fast component (i.e. the marginal obtained when the law of the slow component has been integrated out) by  $\mu_{t}^{\lambda, \mathbf{x}, \mathbf{v}}$. Now we state two lemmas and our main theorem; proofs are postponed. Lemma \ref{lemma:fastProcessConv1} studies the distance between the law of the frozen process and the law of $V_t^{\lambda, \bx, \bv , l}$, while Lemma \ref{lemma:fastProcessConvX} quantifies the distance between the (marginal) law of the fast process,  $\mu_t^{\lambda, \bx,\bv}$, and the limit measure of the frozen process, $\LawLimitPosDep$.
\begin{lemma}\label{lemma:fastProcessConv1}
Let Assumption \ref{ass:driftSlow} and Assumption \ref{ass:p} hold. Then there exists $K>0$ independent of $\mathbf{x} \in \R^N$, $\mathbf{v}\in \fastDomain$, $\lambda$ and $l\coloneqq\{\mathbf{x}^{(k)}\}^{\infty}_{k=1} \subset \bar{B}(\mathbf{x}, (1-z_0)/2K_0\tilde{n})$ such that
\begin{equation}\label{eqn:convMeasXinter}
    \| \fastFrozenLaw_t - \fastProcessSeqContLaw_t \|_{\text{TV}} \leq K \sup_{i}\{\|\mathbf{x}-\mathbf{x}^{(i)}\|_1\},
\end{equation}
where $\|\cdot\|_{TV}$ is the total variation norm, and $\bar{B}(\mathbf{x}, (1-z)/2K_0\tilde{n})$ is the $1$-norm closed ball centred on $\mathbf{x}$ with radius $(1-z)/2K_0\tilde{n}$.
\end{lemma}
\begin{lemma}\label{lemma:fastProcessConvX} Let Assumption \ref{ass:driftSlow} and  Assumption \ref{ass:p} hold.
    Then there exists $K>0$ independent of $\mathbf{x}$, $\mathbf{v}, \lambda$ and $t>0$ such that
\begin{equation}\label{eqn:convMeasX}
    \| \mu_{t}^{\lambda, \mathbf{x}, \mathbf{v}} - \LawLimitPosDep \|_{\text{TV}} \leq K\rate{t\lambda} + Kt,
\end{equation}
where $\LawLimitPosDep$ is as in \eqref{eqn:limMeasX}.
\end{lemma}
\begin{Note}
    Equation \eqref{eqn:convMeasX} does not state 
\begin{equation}\label{eqn:convMeasXnote}
    \| \mu_{t}^{\lambda, \mathbf{x}, \mathbf{v}} - \LawLimitPosDep \|_{\text{TV}}  \rightarrow 0,
\end{equation}
 for fixed $\lambda>0$, as $t \rightarrow \infty$, as was the case in the decoupled regime (see \eqref{eqn:convMeas}). Because of the coupling, in general as $t\rightarrow\infty$, 
$\mu_{t}^{\lambda, \mathbf{x}, \mathbf{v}}$ converges to something other than $\LawLimitPosDep$. In other words, it need not be the case that $\LawLimitPosDep = \lim_{t\rightarrow\infty} \mu_{t}^{\lambda, \mathbf{x}, \mathbf{v}}$.

Instead, one would expect that \eqref{eqn:convMeasXnote} might hold for fixed $t>0$, as $\lambda \rightarrow 0$. This is not implied by \eqref{eqn:convMeasX}, so \eqref{eqn:convMeasX} may not be sharp. While we don't need to show this convergence, we point out that the focus of the proof of Theorem \ref{thm:averaging1} is on working with times $t$ of the order $1/\sqrt{\lambda}$ (see the proof of Theorem \ref{thm:averaging1}), and for these times \eqref{eqn:convMeasX} implies \eqref{eqn:convMeasXnote} for fixed $t>0$, as $\lambda \rightarrow 0$.
\end{Note}

Now we can define the averaged process:
\begin{equation}\label{eqn:avgedProcessX}
    d\bar{X}^{\mathbf{x}, \mathbf{v}}_t = a(\bar{X}^{\mathbf{x}, \mathbf{v}}_t,\zeta^{\mathbf{x},\mathbf{v}})  dt, \quad \bar{X}^{\mathbf{x},\mathbf{v}}_0 = \mathbf{x}.
\end{equation}
where $\zeta^{\mathbf{x},\mathbf{v}}$ is a discrete random variable taking values $\absorbing^{(1)}, \ldots, \absorbing^{(L)}$, such that
    $\zeta^{\mathbf{x},\mathbf{v}} =
    \absorbing^{(i)} \text{ with probability } q^i(\mathbf{x},\mathbf{v})$.
We can write down the expectation of the above process explicitly:
\begin{equation}\label{eqn:avgedSGX}
    \mathbb{E}\left[f(\bar{X}^{\mathbf{x}, \mathbf{v}}_t)\right] = \sum_{i=1}^\numAbs q^i(\mathbf{x},\mathbf{v})\avgSol_t^i f(\mathbf{x}),
\end{equation}
where $\avgSol_t^i f(\mathbf{x})\coloneqq f(y^i(t))$ where $y^i(t)$ is the unique solution of $\frac{d y^i}{d t}=a(y^i(t),\absorbing^{(i)})$ with initial condition $y^i(0) = \mathbf{x}$. Note that formula \eqref{eqn:avgedSG} in Section \ref{sec:non-coupled} is a particular instance of \eqref{eqn:avgedSGX}, and indeed in the setting of Section \ref{sec:non-coupled} one has $\avgSol_t^1 f(\mathbf{x}) = f(\mathbf{x} - t\absorbing)$ and $\avgSol_t^2 f(\mathbf{x}) = f(\mathbf{x} + t\absorbing)$.
\begin{theorem}\label{thm:averaging1}
Let Assumption \ref{ass:driftSlow} hold and let 
  $X^{\lambda,\mathbf{x}, \mathbf{v}}_t$ be as in \eqref{eqn:fullycoupled}, where we assume that the fast process $\fastcoupled$ in \eqref{eqn:fullycoupled} is such that Assumption \ref{ass:p} holds (that is, in short, we are in the setting in which the frozen process has $L$ absorbing states). Let $\bar{X}^{\mathbf{x}, \mathbf{v}}_t$ as in \eqref{eqn:avgedProcessX}. Then there exists $K>0$, possibly dependent on $t$, such that
		\begin{equation*}\label{eqn:conclnonfullycoupledX}
			\left| \mathbb{E}[f(X^{\lambda,\mathbf{x}, \mathbf{v}}_t)] - \mathbb{E}\left[f(\bar{X}^{\mathbf{x}, \mathbf{v}}_t)\right] \right| \leq K\|f\| \frac{1}{\sqrt{\lambda}} , \quad \text{for all $\mathbf{x} \in \R^N$, $\mathbf{v} \in \fastDomain$}
		\end{equation*}
  for all $f\colon \R^N \rightarrow \mathbb{R}$  such that $f \in C^1_b(\R^N)$, 
		where we recall the notation $\|f\| \coloneqq \|f\|_\infty + \sum_{i} \|\partial_{x_i} f\|_{\infty}$.
\end{theorem}
Before proving Lemma \ref{lemma:fastProcessConv1}, Lemma \ref{lemma:fastProcessConvX}, and Theorem \ref{thm:averaging1}, we outline some preliminaries. We begin by pointing out the structure of the transition matrices. In the following we denote by $\transProbsT{\mathbf{x}}$ the transpose of the transition matrix $\transProbs{\mathbf{x}}$, i.e $\transProbsT{\mathbf{x}} \coloneqq \transProbs{\mathbf{x}}^T$. Hence if $\mu_0$ is the initial law of $\fastFrozen_k$, then $\mu_1 = \transProbsT{\mathbf{x}}\mu_0$ is the law of $\fastFrozen_1$. By possibly `reordering' the states in $\fastDomain$ (so that the first $\numAbs$ coincide with the absorbing states) we can write
\begin{equation*}
      \fastProcessMat_{\mathbf{x}} = \left[ {\begin{array}{cc}
    \mathbb{I}_\numAbs & \fastProcessMatAb_{\mathbf{x}} \\
   \mathbf{0}   & \fastProcessMatNAb_{\mathbf{x}}\\
  \end{array} } \right].
\end{equation*}
That is, with this `reordering' the matrix $\fastProcessMat_{\mathbf{x}}$ takes a block structure where $\mathbb{I}_\numAbs \in [0,1]^{\numAbs \times \numAbs}$ is the $\numAbs \times \numAbs$ identity matrix, $\mathbf{0} \in [0,1]^{(2^N - \numAbs) \times \numAbs}$ is the zero matrix, $\fastProcessMatAb_{\mathbf{x}} \in [0,1]^{\numAbs\times (2^N-\numAbs)}$ contains the probabilities of moving from a non-absorbing state to an absorbing state and $\fastProcessMatNAb_{\mathbf{x}} \in [0,1]^{(2^N-\numAbs) \times(2^N-\numAbs)}$ contains the probabilities of moving from a non-absorbing state to a non-absorbing state. Hence,
\begin{equation}\label{eqn:transMatPower}
      \fastProcessMat_{\mathbf{x}}^k = \left[ {\begin{array}{cc}
    \mathbb{I}_\numAbs & \sum_{i=0}^{k-1}\fastProcessMatAb_{\mathbf{x}}\left(\fastProcessMatNAb_{\mathbf{x}}\right)^i\\
    \mathbf{0}   & \left(\fastProcessMatNAb_{\mathbf{x}}\right)^k\\
  \end{array} } \right].
\end{equation}
Note that this is simply the transpose of the transition matrix of the process where $k$ jumps are done in a row. In a similar way, given a sequence $l$, we can write the transpose of the transition matrix of the process $\fastProcessSeq_k$ (which we define to be the discrete-time version of $\fastProcessSeqCont_t$) up to $k$ jumps
\begin{equation}\label{eqn:transMatProd}
      \fastProcessMat_{\mathbf{x}^{(k)}} \cdots \fastProcessMat_{\mathbf{x}^{(1)}} =\prod_{i=0}^{k-1}\fastProcessMat_{\mathbf{x}^{(k-i)}} = \left[ {\begin{array}{cc}
    \mathbb{I}_\numAbs & \sum_{i=0}^{k-1}\fastProcessMatAb_{\mathbf{x}^{(i+1)}}\prod_{j=0}^{i-1}\fastProcessMatNAb_{\mathbf{x}^{(i-j)}}  \\
  \mathbf{0}   & \prod_{i=0}^{k-1}\fastProcessMatNAb_{\mathbf{x}^{(k-i)}}\\
  \end{array} } \right],
\end{equation}
where we use the convention that a product over an empty set is the identity, to ease notation. Using \eqref{eqn:transMatPower} and \eqref{eqn:transMatProd} we have
\begin{equation}\label{eqn:BlockMatIdentity}
\begin{split}
    &\fastProcessMat_{\mathbf{x}}^k - \prod_{i=0}^{k-1}\fastProcessMat_{\mathbf{x}^{(k-i)}} \\ & =\left[ {\begin{array}{cc}
   \mathbf{0}  & \sum_{i=0}^{k-1}\left(\fastProcessMatAb_{\mathbf{x}}\left(\fastProcessMatNAb_{\mathbf{x}}\right)^i - \fastProcessMatAb_{\mathbf{x}^{(i+1)}}\prod_{j=0}^{i-1}\fastProcessMatNAb_{\mathbf{x}^{(i-j)}}\right) \\\mathbf{0} 
     & \left(\fastProcessMatNAb_{\mathbf{x}}\right)^k - \prod_{i=0}^{k-1}\fastProcessMatNAb_{\mathbf{x}^{(k-i)}}\\
  \end{array} } \right].
\end{split}
\end{equation}
   Let $\|\cdot\|_{op}$ be the operator norm of a matrix, namely:
    \begin{equation*}
        \|A\|_{op} = \sup_{\|u\|_1=1}\|Au\|_1.
    \end{equation*}
 By Assumption \ref{ass:p} \eqref{ass:pMinor}
    \begin{equation}\label{eqn:opMatBound}
        \|\left(\fastProcessMatNAb_{\mathbf{x}}\right)^{\tilde{n}}\|_{op} \leq z_0 < 1.
        \end{equation}
        More precisely, 
        \begin{equation*}
            \|\left(\fastProcessMatNAb_{\mathbf{x}}\right)^{\tilde{n}}\|_{op} = \max_{\mathbf{v} \notin \{\absorbing^{(1)}, \ldots,\absorbing^{(\numAbs)}\}}\mathbb{P}[\fastFrozenDisc_{\tilde{n}} \notin \{\absorbing^{(1)}, \ldots,\absorbing^{(\numAbs)}\}]\leq z_0 < 1.
        \end{equation*}
        We also define
        \begin{equation}\label{eqn:opMatBound1}
    C^{Ab} \coloneqq \sup_{\mathbf{x}}\|\fastProcessMatAb_{\mathbf{x}}\|_{op}.
    \end{equation}
We also have, by Assumption \ref{ass:p} \eqref{ass:pLip} that for all $\mathbf{x},\mathbf{x'} \in \R^N$,
    \begin{equation}\label{eqn:opMatBoundDiff}
        \|\fastProcessMatNAb_{\mathbf{x}}- \fastProcessMatNAb_{\mathbf{x'}}\|_{op} + \|\fastProcessMatAb_{\mathbf{x}}- \fastProcessMatAb_{\mathbf{x'}}\|_{op}\leq K_0\|\mathbf{x}-\mathbf{x'}\|_1.
    \end{equation}

\begin{proof}[Proof of Lemma \ref{lemma:fastProcessConv1}]

    First we consider the matrix \eqref{eqn:BlockMatIdentity}. We wish to bound the operator norm of such a matrix.
    We can write the difference in the bottom right hand block of \eqref{eqn:BlockMatIdentity} as
\begin{equation}\label{eqn:MatIdentity}
    \left(\fastProcessMatNAb_{\mathbf{x}}\right)^k - \prod_{i=0}^{k-1}\fastProcessMatNAb_{\mathbf{x}^{(k-i)}} = \sum_{j=0}^{k-1}\left(\prod_{i=0}^{j-1}\fastProcessMatNAb_{\mathbf{x}^{(k-i)}}\right)\left(\fastProcessMatNAb_{\mathbf{x}}-\fastProcessMatNAb_{\mathbf{x}^{(k-j)}} \right)\left(\fastProcessMatNAb_{\mathbf{x}} \right)^{k-1-j}.
\end{equation}
Using \eqref{eqn:opMatBound}, \eqref{eqn:opMatBoundDiff} and \eqref{eqn:MatIdentity} we then have \begin{equation*}
     \Big\|\left(\fastProcessMatNAb_{\mathbf{x}}\right)^{\tilde{n}} - \prod_{i=0}^{\tilde{n}-1}\fastProcessMatNAb_{\mathbf{x}^{(\tilde{n}-i)}}\Big\|_{op} \leq K_0 \tilde{n} \|\mathbf{x}-\mathbf{x}^{(i)}\|_1.
\end{equation*}
Hence, using \eqref{eqn:opMatBound} yields, for all subsequences $\{\mathbf{x}^{(n_i)}\}^{\infty}_{i=0}$ of $l$
\begin{equation*}
    \|\prod_{i=0}^{\tilde{n}-1}\fastProcessMatNAb_{\mathbf{x}^{(\tilde{n}-n_i)}}\|_{op} \leq z_0 +  K_0 \tilde{n} \max_{1\leq i\leq k} \|\mathbf{x}-\mathbf{x}^{(n_i)}\|_1 \leq \frac{z_0+1}{2} \eqqcolon \tilde{z} <1.
\end{equation*}
Using \eqref{eqn:opMatBound}, \eqref{eqn:opMatBoundDiff} and \eqref{eqn:MatIdentity} we have
\begin{equation}\label{eqn:diffBoundPowProd}\begin{split}
    \Big\|\left(\fastProcessMatNAb_{\mathbf{x}}\right)^k - \prod_{i=0}^{k-1}\fastProcessMatNAb_{\mathbf{x}^{(k-i)}}\Big\|_{op} &\leq C\sum_{j=0}^{k-1}\tilde{z}^{\lfloor \frac{k-1}{\tilde{n}}\rfloor}\|\mathbf{x}-\mathbf{x}^{(k-j)}\|_1 \\
    &\leq \tilde{C}\cdot k \cdot \tilde{z}^{\lfloor \frac{k-1}{\tilde{n}}\rfloor}\max_{1\leq i\leq k} \|\mathbf{x}-\mathbf{x}^{(i)}\|_1
\end{split}
\end{equation}
for some $\tilde{C}>0$ independent of $k$ and $\mathbf{x}$. Now we consider the upper right block:
\begin{equation}\label{eqn:bottLeftDiff}
    \begin{split}
        \sum_{i=0}^{k-1}&\left(\fastProcessMatAb_{\mathbf{x}}\left(\fastProcessMatNAb_{\mathbf{x}}\right)^i - \fastProcessMatAb_{\mathbf{x}^{(i+1)}}\prod_{j=0}^{i-1}\fastProcessMatNAb_{\mathbf{x}^{(i-j)}}\right) \\
        &\leq \sum_{i=0}^{k-1}\left( \fastProcessMatAb_{\mathbf{x}}
- \fastProcessMatAb_{\mathbf{x}^{(i+1)}} \right)\left(\fastProcessMatNAb_{\mathbf{x}}\right)^i
+\sum_{i=0}^{k-1}\fastProcessMatAb_{\mathbf{x}^{(i+1)}} \left(\left(\fastProcessMatNAb_{\mathbf{x}}\right)^i-\prod_{j=0}^{i-1}\fastProcessMatNAb_{\mathbf{x}^{(i-j)}}\right)
    \end{split}
\end{equation}
Using \eqref{eqn:opMatBound}, \eqref{eqn:opMatBound1}, \eqref{eqn:diffBoundPowProd} and \eqref{eqn:bottLeftDiff}
\begin{equation}\label{eqn:bottLeftDiffBound}
    \begin{split}
        \Big\|\sum_{i=0}^{k-1}&\left(\fastProcessMatAb_{\mathbf{x}}\left(\fastProcessMatNAb_{\mathbf{x}}\right)^i - \fastProcessMatAb_{\mathbf{x}^{(i+1)}}\prod_{j=0}^{i-1}\fastProcessMatNAb_{\mathbf{x}^{(i-j)}}\right)\Big\|_{op} \\
        &\leq C\sum_{i=0}^{k-1} z_0^{\lfloor \frac{i}{\tilde{n}}\rfloor} \|\mathbf{x}-\mathbf{x}^{(i+1)}\|_1 + \tilde{C}\cdot C^{Ab}\sum_{i=0}^{k-1}i \cdot \tilde{z}^{\lfloor \frac{i-1}{\tilde{n}}\rfloor}\max_{1\leq j\leq i} \|\mathbf{x}-\mathbf{x}_{(j)}\|_1 \\
        &\leq C \max_{1\leq i\leq k} \|\mathbf{x}-\mathbf{x}^{(i)}\|_1
    \end{split}
\end{equation}
for some $C>0$ independent of $k$. Hence, by \eqref{eqn:BlockMatIdentity}, \eqref{eqn:diffBoundPowProd} and \eqref{eqn:bottLeftDiffBound}
\begin{equation}\label{eqn:matBound}
    \|\fastProcessMat_{\mathbf{x}}^k - \prod_{i=0}^{k-1}\fastProcessMat_{\mathbf{x}^{(k-i)}}\|_{op} \leq C\max_{1\leq i\leq k} \|\mathbf{x}-\mathbf{x}^{(i)}\|_1.
\end{equation}
Now \eqref{eqn:convMeasXinter} follows from \eqref{eqn:matBound} by writing (where we denote by $\fastProcessSeqLaw_k$ the law of $\fastProcessSeq_k$, the discrete version of $\fastProcessSeqCont_t$)
\begin{equation*}
\begin{split}
         \| \fastFrozenLaw_t - \fastProcessSeqContLaw_t \|_{\text{TV}} &= \| \sum_{k=0}^{\infty} \mathbb{P} \left[\tau^\lambda(t) = k\right](\fastFrozenDiscLaw_k - \fastProcessSeqLaw_k) \|_{\text{TV}} \\
         &\leq  \sum_{k=0}^{\infty} \mathbb{P} \left[\tau^\lambda(t) = k\right]\|\fastFrozenDiscLaw_k - \fastProcessSeqLaw_k \|_{TV} \\
         &\leq \sum_{k=0}^{\infty} \mathbb{P} \left[\tau^\lambda(t) = k\right] \|\fastProcessMat_{\mathbf{x}}^k - \prod_{i=0}^{k-1}\fastProcessMat_{\mathbf{x}^{(k-i)}}\|_{op} \\
         &\leq C \sup_{ i} \|\mathbf{x}-\mathbf{x}^{(i)}\|_1.
\end{split}
\end{equation*}
\end{proof}

\begin{proof}[Proof of Lemma \ref{lemma:fastProcessConvX}]
Firstly,
\begin{equation}\label{eqn:triangleIneq}
    \| \mu_{t}^{\lambda, \mathbf{x}, \mathbf{v}} - \LawLimitPosDep \|_{\text{TV}} \leq \| \mu_{t}^{\lambda, \mathbf{x}, \mathbf{v}} - \fastFrozenLaw_t \|_{\text{TV}} + \| \fastFrozenLaw_t - \LawLimitPosDep \|_{\text{TV}}.
\end{equation}
From \eqref{eqn:convMeasXinter} we have
    \begin{equation}\label{eqn:firstTermTriangle}
    \begin{split}
         \| \mu_{t}^{\lambda, \mathbf{x}, \mathbf{v}} - \fastFrozenLaw_t \|_{\text{TV}} &\leq \sup_{l\in \seqSpace} \| \fastFrozenLaw_t - \fastProcessSeqContLaw_t \|_{\text{TV}} \\
         &\leq C \sup_{\mathbf{x'} \in \bar{B}(\mathbf{x},\|a\|_{\infty} t)} \|\mathbf{x}-\mathbf{x'}\|_1 \\
         &\leq Ct,
    \end{split}
    \end{equation} where $\seqSpace$ is the space of sequences $(\mathbf{x}^{(k)})^\infty_{k=1}$ with terms $\mathbf{x}^{(k)} \in \bar{B}(\mathbf{x},\|a\|_{\infty} t)$ and $\bar{B}(\mathbf{x},\|a\|_{\infty} t)$ is the $1$-norm closed ball centred on $\mathbf{x}$ with radius $\|a\|_{\infty} t$. To see the first step of this inequality notice that at time $t$ the slow dynamics is contained in $\bar{B}(\mathbf{x},\|a\|_{\infty} t)$, so all terms of the sequence of positions that the Poisson clock goes off at are contained in $\bar{B}(\mathbf{x},\|a\|_{\infty} t)$.
    Finally, \eqref{eqn:convMeasX} follows from 
    \eqref{eqn:convMeasXX}, \eqref{eqn:triangleIneq} and \eqref{eqn:firstTermTriangle}.
\end{proof}

\begin{proof}[Proof of Theorem \ref{thm:averaging1}.]
     The proof is analogous in spirit to the proof of Proposition \ref{thm:averaging}. We begin by defining $t_0 \coloneqq \frac{1}{\sqrt{\lambda}}$.
    Instead of \eqref{eqn:decomp}, using \eqref{eqn:avgedSGX} we have
    \begin{equation*}\label{eqn:conditioningX}
        \begin{split}
            &\left|\mathbb{E}\big[f\big(X^{\lambda,\mathbf{x}, \mathbf{v}}_t \big) \big]-  \mathbb{E}\left[f(\bar{X}^{\mathbf{x},\mathbf{v}}_t)\right]\right|\\ &\leq \sum^{\numAbs}_{i=1}|\mathbb{P}[ V^{\lambda,\mathbf{x}, \mathbf{v}}_{t_0} =\absorbing^{(i)}]\mathbb{E}\big[f\big(X^{\lambda,\mathbf{x}, \mathbf{v}}_t \big) | V^{\lambda,\mathbf{x}, \mathbf{v}}_{t_0} = \absorbing^{(i)}\big] - q^i(\mathbf{v})\avgSol_t^i f(\mathbf{x})| \\
            &+ \mathbb{P}[ V^{\lambda,\mathbf{x}, \mathbf{v}}_{t_0} \notin \{\absorbing^{(1)},\ldots, \absorbing^{(\numAbs)}\}]|\mathbb{E}\big[f\big(X^{\lambda,\mathbf{x}, \mathbf{v}}_t \big) | V^{\lambda,\mathbf{x}, \mathbf{v}}_{t_0} \notin \{\absorbing^{(1)},\ldots, \absorbing^{(\numAbs)}\}\big]|.
        \end{split}
    \end{equation*}
    Now, instead of \eqref{eqn:1}-\eqref{eqn:3} we proceed by proving the existence of a constant $K>0$ such that the following two bounds hold
    \begin{enumerate}
\item For every $1\leq i \leq \numAbs$
\begin{equation}\label{eqn:1X}
    |A_i - q^i(\mathbf{x},\mathbf{v})\avgSol_t^i f(\mathbf{x})| \leq K \|f\|\frac{1}{\sqrt{\lambda}},
\end{equation}
where $A_i =\mathbb{P}[ V^{\lambda,\mathbf{x}, \mathbf{v}}_{t_0} =\absorbing^{(i)}]\mathbb{E}\big[f\big(X^{\lambda,\mathbf{x}, \mathbf{v}}_t \big) | V^{\lambda,\mathbf{x}, \mathbf{v}}_{t_0} = \absorbing^{(i)}\big]$.
\item We have
\begin{equation}\label{eqn:3X}
      |C| \leq K \frac{1}{\sqrt{\lambda}}\|f\|,
\end{equation}
where $C\coloneqq \mathbb{P}[ V^{\lambda,\mathbf{x}, \mathbf{v}}_{t_0} \notin \{\absorbing^{(1)},\ldots, \absorbing^{(\numAbs)}\}]\mathbb{E}\big[f\big(X^{\lambda,\mathbf{x}, \mathbf{v}}_t \big) | V^{\lambda,\mathbf{x}, \mathbf{v}}_{t_0} \notin \{\absorbing^{(1)},\ldots, \absorbing^{(\numAbs)}\}\big]$
\end{enumerate}
Equation \eqref{eqn:1X} is shown by first writing \begin{equation}\label{eqn:ergRewriteExp}
    \mathbb{E}\big[f\big(X^{\lambda,\mathbf{x}, \mathbf{v}}_t \big) | V^{\lambda,\mathbf{x}, \mathbf{v}}_{t_0} = \absorbing^{(i)}\big] = \mathbb{E}[\avgSol_{t-t_0}^i f(X^{\lambda,\mathbf{x}, \mathbf{v}}_{t_0})], \quad t>t_0
\end{equation}
so that a similar procedure to \eqref{eqn:AqN} can be followed. That is, we write
\begin{equation}\label{eqn:AqNX}
    \begin{split}
        |A_i - q^i(\mathbf{x},\mathbf{v})\avgSol_t^i f(\mathbf{x})| &\leq | \mathbb{E}[\avgSol_{t-t_0}^i f(X^{\lambda,\mathbf{x}, \mathbf{v}}_{t_0})] -\avgSol_t^i f(\mathbf{x})  | \\
        &+ | \avgSol_t^i f(\bx)||\mathbb{P}[ V^{\lambda,\mathbf{x}, \mathbf{v}}_{t_0} =\absorbing^{(i)}] - q^i(\bx,\bv)|.
    \end{split}
\end{equation}
Instead of \eqref{eqn:D} we use Assumption \ref{ass:driftSlow} to address the first term of \eqref{eqn:AqNX} to get
\begin{equation}\label{eqn:D1}
\begin{split}
    | \mathbb{E}[\avgSol_{t-t_0}^i f(X^{\lambda,\mathbf{x}, \mathbf{v}}_{t_0})] -\avgSol_t^i f(\mathbf{x})  |&\leq\sup_{\mathbf{x}'\in \bar{B}(\mathbf{x}, \|a\|_{\infty}t_0)}\left|\avgSol_{t-t_0}^i f(\mathbf{x}') -\avgSol_t^i f(\mathbf{x})\right| \\ &\leq  K\|f\|\cdot \frac{1}{\sqrt{\lambda}}.
\end{split}
\end{equation} Instead of \eqref{eqn:fLeqEpsilon} we have that there exists $\tilde{K}>0$ such that \begin{equation}\label{eqn:fLeqEpsilon1}\begin{split}
       | \avgSol_t^i f(\bx)||\mathbb{P}[ V^{\lambda,\mathbf{x}, \mathbf{v}}_{t_0} =\absorbing^{(i)}] - q^i(\bx,\bv)| &\leq  | \avgSol_t^i f(\bx)|| \mu_{t}^{\lambda, \mathbf{x}, \mathbf{v}} - \LawLimitPosDep \|_{\text{TV}} \\
        &\leq  | \avgSol_t^i f(\bx)|(K\rate{t_0\lambda} + Kt_0) \\ &\leq \tilde{K}\|f\|_\infty\frac{1}{\sqrt{\lambda}} ,
\end{split}
\end{equation}
having used Lemma \ref{lemma:fastProcessConvX} and Assumption \ref{ass:driftSlow}. Hence, by \eqref{eqn:D1} and \eqref{eqn:fLeqEpsilon1}, \eqref{eqn:1X} follows. Equation \eqref{eqn:3X} is shown in the same way as \eqref{eqn:3}, using Lemma \ref{lemma:fastProcessConvX}.
This concludes the proof.
\end{proof}

\subsection{Frozen process with multiple ergodic classes}\label{sec:ergClasses}
Now we extend the result of Theorem \ref{thm:averaging1} to the case where, instead of absorbing \emph{states}, we have ergodic \emph{classes}. That is, each $\absorbing^{(i)}$ is now a set of states $E^{(i)} = \{\absorbingErg{(i)}{1},\ldots,\absorbingErg{(i)}{m_i}\}$, within which, the frozen process is ergodic in the sense that there is one unique invariant measure $\mu^{(i)}(\bx; \cdot)$ if we restrict to this set.


\begin{assumption}\label{ass:pErg}
The following hold for the function $\transProbs{}\colon \R^N \times \fastDomain\times\fastDomain \rightarrow[0,1]$
\begin{enumerate}[(i)]
\item \label{ass:pProbErg} Transition probabilities: For all $\mathbf{x}\in \R^N$ and $\mathbf{v} \in \fastDomain$, we have \\$\sum_{\mathbf{v}' \in \fastDomain} \transProbs{\mathbf{x}}(\mathbf{v},\mathbf{v}')=1$.
        \item  The sets $E^{(1)}, \ldots, E^{(L)}$ are ergodic classes: \label{ass:p:AbsErg} For all $1\leq i\leq\numAbs, \mathbf{x}\in \R^N\text{ and }1\leq j\leq m_i$ we have $
     \sum^{m_i}_{j'=1}\transProbs{\mathbf{x}}( \absorbingErg{(i)}{j}, \absorbingErg{(i)}{j'}) = 1$. Furthermore for every $\bv \in E^{(i)}$ there exists $n' \geq 1$, which is independent of $\bv$, such that $$ \mathbb{P}[\fastFrozenDisc_{n'} = \absorbingErg{(i)}{j}] > 0, \quad \text{for every $1\leq j\leq m_i$}.$$
        \item \label{ass:pMinorErg} Probability of absorption: There exist $\tilde{n}\geq1$ and $z_0<1$ such that for all $\bx \in \R^N, \bv \in \fastDomain$ \begin{equation*}
            \mathbb{P}[\fastFrozenDisc_{\tilde{n}} \notin \{E^{(1)} \cup \ldots \cup E^{(L)}\}] \leq z_0 .
        \end{equation*}
        \item \label{ass:pLipErg}  Lipschitzianity in $\bx$: There exists $K_0>0$ such that for all $\mathbf{v} \in \fastDomain$ and $\mathbf{x},\mathbf{x'} \in \R^N$,
    \begin{equation*}
       \sum_{\mathbf{v'} \in \fastDomain} |\transProbs{\mathbf{x}}( \mathbf{v}, \mathbf{v'}) - \transProbs{\mathbf{x'}} (\mathbf{v}, \mathbf{v'}) | \leq K_0\|\mathbf{x}-\mathbf{x'}\|_1
    \end{equation*}
    where $\|\mathbf{w}\|_1 = \sum^n_{i=1}|w_i|$ for any $\mathbf{w} \in \R^{N}$.
    \end{enumerate}
\end{assumption}
\begin{Note} In the case where the ergodic classes are single states, i.e $E^{(i)} = \absorbing^{(i)}$, Assumption \ref{ass:pErg} is equivalent to Assumption \ref{ass:p}. Assumption \ref{ass:pErg} \eqref{ass:p:AbsErg} implies the existence of the unique invariant measure $\mu^{(i)}(\bx; \cdot)$ for each $i=1,\ldots, L$ and $\bx\in\R^N$.
\end{Note}
    
    The proof of this result is very similar to Theorem \ref{thm:averaging1}, but we first introduce some new notation. Instead of \eqref{eqn:limMeasX} we have
\begin{equation*}\label{eqn:limMeasXErg}
    \LawLimitPosDep(\mathbf{v'}) \coloneqq \begin{cases}
    q^{(i)}(\bx,\bv)\cdot \mu^{(i)}(\bx;\bv ') & \text{if } \mathbf{v'} = \absorbingErg{(i)}{j} \text{ for some }1\leq i\leq \numAbs, \, 1\leq j\leq m_i \\
     0 & \text{otherwise,}
    \end{cases}
\end{equation*}
where $q^{(i)}(\mathbf{x},\mathbf{v})$ is the probability that $\fastFrozen_t$ enters $\absorbing^{(i)}$; again, see, for example, \cite[Chapter 3]{norris1998markov} for details on how to calculate $q^{(i)}(\mathbf{x},\mathbf{v})$ using the transition probabilities. Note that for fixed $\bx, \bv$, $\LawLimitPosDep$ is a probability measure since $$\sum_{\bv' \in \fastDomain} \LawLimitPosDep(\mathbf{v'})= \sum_{1\leq i \leq L}\sum_{\bv' \in E^{(i)}} q^{(i)}(\bx,\bv)\cdot \mu^{(i)}(\bx;\bv ')=1.$$ Further, note that the existence of $q^{(i)}(\mathbf{x},\mathbf{v}')$ and the convergence of $\fastFrozenLaw_t$ to $\LawLimitPosDep$ for fixed $t>0$ as $\lambda \rightarrow \infty$ is implied by Assumption \ref{ass:pErg}.
Now we can define the averaged process $\bar{X}^{\mathbf{x}, \mathbf{v}}_t$, namely:
\begin{equation}\label{eqn:avgedProcessXErg}
    d\bar{X}^{\mathbf{x}, \mathbf{v}}_t = \sum_{j=1}^{m_{\zeta^{\mathbf{x},\mathbf{v}}}}\mu^{(\zeta^{\mathbf{x},\mathbf{v}})}(\bx; \absorbingErg{(\zeta^{\mathbf{x},\mathbf{v}})}{j})a(\bar{X}^{\mathbf{x}, \mathbf{v}}_t,\absorbingErg{(\zeta^{\mathbf{x},\mathbf{v}})}{j})  dt, \quad \bar{X}^{\mathbf{x},\mathbf{v}}_0 = \mathbf{x},
\end{equation}
where $\zeta^{\mathbf{x},\mathbf{v}}$ is a random variable, independent of time, with possible values $1, \ldots,L$ such that
    $\zeta^{\mathbf{x},\mathbf{v}} =
   i \text{ with probability } q^{(i)}(\mathbf{x},\mathbf{v})$.

The expectation of $\bar{X}^{\mathbf{x}, \mathbf{v}}_t$ is the same as in \eqref{eqn:avgedSGX} which we write out here for the convenience of the reader
\begin{equation}
    \mathbb{E}\left[f(\bar{X}^{\mathbf{x}, \mathbf{v}}_t)\right] = \sum_{i=1}^\numAbs q^i(\mathbf{x},\mathbf{v})\avgSol_t^i f(\mathbf{x}),
\end{equation}
except this time $\avgSol_t^i f(\mathbf{x})\coloneqq f(y^i(t))$ where $y^i(t)$ is the unique solution of the equation 
$$\frac{d y^i}{d t}=\sum_{j=1}^{m_{i}}\mu^{(i)}(\mathbf{x};\absorbingErg{(i)}{j})a(\bar{X}^{\mathbf{x}, \mathbf{v}}_t,\absorbingErg{(i)}{j}),$$
with initial condition $y^i(0) = \mathbf{x}$.

\begin{theorem}\label{thm:averaging1Erg}
Let Assumption \ref{ass:driftSlow} hold and let 
  $X^{\lambda,\mathbf{x}, \mathbf{v}}_t$ be as in \eqref{eqn:fullycoupled}, where we assume that the fast process $\fastcoupled$ in \eqref{eqn:fullycoupled} is such that Assumption \ref{ass:pErg} holds (that is, in short, we are in the setting in which the frozen process has $L$ absorbing classes). Let $\bar{X}^{\mathbf{x}, \mathbf{v}}_t$ as in \eqref{eqn:avgedProcessXErg}. Then there exists $K>0$, possibly dependent on $t$, such that
		\begin{equation}\label{eqn:conclnonfullycoupledXErg}
			\left| \mathbb{E}[f(X^{\lambda,\mathbf{x}, \mathbf{v}}_t)] - \mathbb{E}\left[f(\bar{X}^{\mathbf{x}, \mathbf{v}}_t)\right] \right| \leq K\|f\| \frac{1}{\sqrt{\lambda}}, \quad \text{ for all $\mathbf{x} \in \R^N$, $\mathbf{v} \in \fastDomain$}.
		\end{equation}
\end{theorem}

\section*{Acknowledgments}
 KJP is a member of INdAM-GNFM and acknowledges "Miur-Dipartimento di Eccellenza" funding to the Dipartimento di Scienze, Progetto e Politiche del Territorio (DIST). M.O. acknowledges support from the Leverhulme grant  RPG-2020-09. I.S was supported by the EPSRC Centre for Doctoral Training in Mathematical 
Modelling, Analysis and Computation (MAC-MIGS) funded by the UK Engineering and 
Physical Sciences Research Council (grant EP/S023291/1), Heriot-Watt University and the 
University of Edinburgh. We thank the anonymous referees for thoughtful comments, which helped improve the paper. 

\begin{appendix}
\section{Auxiliary proofs}\label{sec:AuxProofs}
\begin{proof}[Proof of Lemma \ref{lemma:fastProcessConv}]
We begin by proving \eqref{eqn:convAbs0}. In the following, we denote by $\{V^{\mathbf{v}}(k)\}^\infty_{k=0}$ the discrete time version of the Markov chain $V^{\lambda, \mathbf{v}}_t$.
That is, $\{V^{\mathbf{v}}(k)\}^\infty_{k=0}$ is the process that changes value at every time-unit rather than when the Poisson clock goes off according to the transition probabilities $\transProbs{}(\mathbf{v}, \mathbf{v}')$ of going from state $\mathbf{v} \in \{-1,+1\}^N$ to $\mathbf{v}' \in \{-1,+1\}^N$ (construction of this discrete time-process is analogous to \eqref{eqn:condPoissonProcess}).

There is at least one way of reaching $\mathbf{e}$ or $-\mathbf{e}$ after $N$ steps of the discrete time Markov chain $V^{\mathbf{v}}(k)$ and hence the probability that the chain has not been absorbed after $N$ steps is strictly less than $1$ (i.e $\mathbb{P}\left[V^{\mathbf{v}}(N) \notin \{\absorbing,-\absorbing\} \right] \leq b_\mathbf{v}<1$). Since $\{-1,+1\}^N$ is finite, there exists $b<1$ independent of $\mathbf{v}$ (but dependent on $N$) such that
\begin{equation*}
   \mathbb{P}\left[V^{\mathbf{v}}(N) \notin \{\absorbing,-\absorbing\} \right] \leq b.
\end{equation*}
Hence, for all $k$,
\begin{equation}\label{eqn:fastDecay}
   \mathbb{P}\left[V^{\mathbf{v}}(k) \notin \{\absorbing,-\absorbing\} \right]  \leq \left( \mathbb{P}\left[V^{\mathbf{v}}(N) \notin \{\absorbing,-\absorbing\} \right]\right)^{\lfloor\frac{k}{N}\rfloor} \leq b^{\lfloor\frac{k}{N}\rfloor}.
\end{equation}
From here, we assume that $b>0$, otherwise we are done.
Using \eqref{eqn:fastDecay} and the distribution of the Poisson process $\tau_t^{N\lambda}$
\begin{equation}\label{eqn:psnEst}
    \begin{split}
        \mathbb{P}[V^{\lambda, \mathbf{v}}_t \notin \{\absorbing,-\absorbing\}] &= \sum_{k=0}^{\infty} \mathbb{P} \left[\tau_t^{N\lambda} = k\right] \mathbb{P}\left[V^{\mathbf{v}}(k) \notin \{\absorbing,-\absorbing\} \right] \\
        &\leq \sum_{k=0}^{\infty} \mathbb{P} \left[\tau_t^{N\lambda} = k\right]b^{\lfloor\frac{k}{N}\rfloor} \leq \sum_{k=0}^{\infty} \frac{(N\lambda t)^k e^{-N\lambda t}}{k!} b^{\lfloor\frac{k}{N}\rfloor}.\end{split}
        \end{equation}
        Then, since $\lfloor\frac{k}{N} \rfloor \geq \frac{k}{N}- \frac{N-1}{N} $ we can write
        \begin{equation} \label{eqn:sumEst}
            \begin{split}
       \sum_{k=0}^{\infty} \frac{(N\lambda t)^k e^{-N\lambda t}}{k!} b^{\lfloor\frac{k}{N}\rfloor} &\leq b^{-\frac{N-1}{N}} e^{-N\lambda t}\sum_{k=0}^{\infty} \frac{(N\lambda t b^{1/N})^k }{k!} \\
        &\leq b^{-\frac{N-1}{N}} e^{-N\lambda t(1-b^{1/N})}.
    \end{split}
\end{equation}
Then \eqref{eqn:psnEst} and \eqref{eqn:sumEst} prove \eqref{eqn:convAbs0}.

We use \eqref{eqn:convAbs0} to show that \eqref{eqn:convAbs1} and \eqref{eqn:convAbs2} must hold. We first use \eqref{eqn:convAbs0} to write
\begin{equation*}
    1- b^{-\frac{N-1}{N}} e^{-N\lambda t(1-b^{1/N})} \leq \mathbb{P}[V^{\lambda, \mathbf{v}}_t = -\absorbing] - q(\mathbf{v}) + \mathbb{P}[V^{\lambda, \mathbf{v}}_t = \absorbing] -(1- q(\mathbf{v})) +1  \leq 1.
\end{equation*}
Hence,
\begin{equation*}
    0 \leq  q(\mathbf{v}) - \mathbb{P}[V^{\lambda, \mathbf{v}}_t = -\absorbing] + (1- q(\mathbf{v})) -\mathbb{P}[V^{\lambda, \mathbf{v}}_t = \absorbing] \leq K e^{-N\lambda t(1-b^{1/N})}.
\end{equation*}
Now we note that $\mathbb{P}[V^{\lambda, \mathbf{v}}_t = \absorbing]$ and $\mathbb{P}[V^{\lambda, \mathbf{v}}_t = -\absorbing]$ are increasing functions of time. This can be seen by observing that, as you increase the number of available steps to get to $\mathbf{e}$ or $-\mathbf{e}$, the number of ways to get to  $\mathbf{e}$ or $-\mathbf{e}$ increases as well (though not strictly). Hence, \eqref{eqn:convAbs1} and \eqref{eqn:convAbs2} are shown.

To complete the proof, \eqref{eqn:convMeas} follows from \eqref{eqn:convAbs0}-\eqref{eqn:convAbs2} since for all $A$ in the power set of $\{-1,+1\}^N$ we have
\begin{equation*}
\begin{split}
       |\mu_{t}^{\lambda, \mathbf{v}}(A) - \mu^{\mathbf{v}}_{\infty}(A)| &\leq  |\mu_{t}^{\lambda, \mathbf{v}}(A\setminus \{\absorbing,-\absorbing\})|
+|\mu_{t}^{\lambda, \mathbf{v}}(\{-\absorbing\}) - q(\mathbf{v})| \\&+
|\mu_{t}^{\lambda, \mathbf{v}}(\{\absorbing\}) - (1-q(\mathbf{v}))| \\
&\leq 3K_1 e^{-c_1t\lambda}.
\end{split}
\end{equation*}
\end{proof}

\begin{proof}[Proof of Theorem \ref{thm:averaging1Erg}]
    The proof of Theorem  is very similar to that of Theorem \ref{thm:averaging1},   so we just sketch it and point out the differences. We cannot write \eqref{eqn:ergRewriteExp}, since the process is exactly the averaged process once it is absorbed into an ergodic class. However, we can use classical averaging results to work around this by writing
    \begin{equation*}
    \begin{split}
         &\left|\mathbb{E}\big[f\big(X^{\lambda,\mathbf{x}, \mathbf{v}}_t \big) | V^{\lambda, \mathbf{x},\mathbf{v}}_{t_0} \in E^{(i)}\big] - \avgSol_t^i f(\mathbf{x})\right| \leq \\
         &\left|\mathbb{E}\big[\mathbb{E}\big[f\big(X^{\lambda,\mathbf{x}, \mathbf{v}}_{t} \big) -\avgSol_{t-t_0}^i f(X^{\lambda,\mathbf{x}, \mathbf{v}}_{t_0})| V^{\lambda,\mathbf{x}, \mathbf{v}}_{t_0} \in E^{(i)} ,X^{\lambda,\mathbf{x}, \mathbf{v}}_{t_0}\big]| V^{\lambda,\mathbf{x}, \mathbf{v}}_{t_0}\in E^{(i)} \big]\right|
         \\
         &+\sup_{\mathbf{x}'\in \bar{B}(\mathbf{x}, \|a\|_{\infty}t_0)}\left|\avgSol_{t-t_0}^i f(\mathbf{x}') -\avgSol_t^i f(\mathbf{x})\right|,
    \end{split}
    \end{equation*}
    where $t_0 = \frac{1}{\lambda}$.
 The first addend can be bounded via classical averaging techniques,  see e.g \cite[Chapter $16$]{pavliotis2008multiscale}, using which it is easy to show that the first addend tends to zero as $\lambda \rightarrow \infty$.
 
 The second addend can be bounded using the Lipschitzianity $f$ and Assumption \ref{ass:driftSlow}. That is, we have \begin{equation*}
     \sup_{\mathbf{x}'\in \bar{B}(\mathbf{x}, \|a\|_{\infty}t_0)}\left|\avgSol_{t-t_0}^i f(\mathbf{x}') -\avgSol_t^i f(\mathbf{x})\right| \leq K \|f \|\frac{1}{\sqrt{\lambda}}.
 \end{equation*}
Hence a bound analogous to \eqref{eqn:1X} follows again by Lemma \ref{lemma:fastProcessConvX}. A bound analogous to \eqref{eqn:3X} can be proved in exactly the same way as in Theorem \ref{thm:averaging1}. Hence \eqref{eqn:conclnonfullycoupledXErg} is shown.
\end{proof}
\end{appendix}

\bibliographystyle{ieeetr}
\bibliography{arxivresubmission}

\end{document}